\documentclass[a4paper]{article}
\usepackage[T1]{fontenc}
\usepackage[utf8]{inputenc}
\usepackage{framed}
\usepackage[english]{babel}
\usepackage{amssymb}
\usepackage{amsthm}
\usepackage{amsmath}
\usepackage{stmaryrd,mathtools}
\mathtoolsset{showonlyrefs}
\usepackage{bbm}
\usepackage[titletoc,title]{appendix}
\usepackage[bbgreekl]{mathbbol}
\usepackage{tikz}
\usepackage{marginnote}
\usepackage{tikz-cd}
\usetikzlibrary{matrix,arrows,decorations.pathmorphing}
\usepackage{graphicx}
\usepackage[footnote,draft,danish,silent,nomargin]{fixme}
\usepackage[bookmarks,colorlinks]{hyperref}
\usepackage{enumerate}

\theoremstyle{plain}
\newtheorem*{Theorem*}{Theorem}
\newtheorem*{definition*}{Definition}
\newtheorem{theorem}{Theorem}[section]
\newtheorem{definition}{Definition}[section]
\newtheorem{remark}{Remark}

\newtheorem*{conjecture*}{Conjecture}

\newtheorem*{ftheorem*}{Main Theorem} 
\newtheorem{lemma}{Lemma}[section]
\newtheorem*{lemma*}{Lemma}
\newtheorem{proposition}{Proposition}[section]
\newtheorem{proposition-definition}{Proposition-Definition}[section]
\newtheorem*{proposition*}{Proposition}
\newtheorem{corollary}{Corollary}[section]
\newtheorem{corollary*}{Corollary}

\newtheorem*{ex*}{Example}

\newtheorem*{pr*}{Problem}

\newcommand{\R}{\mathbb{R}}
\newcommand{\Q}{\mathbbm{Q}}

\newcommand{\C}{\mathbbm{C}}

\newcommand{\LD}{\mathcal{L}}
\newcommand{\n}{\mathcal{N}}
\newcommand{\M}{\mathcal{M}}

\newcommand{\Z}{\mathbb{Z}}
\newcommand{\To}{\mathbb{T}}

\DeclareMathOperator{\tr}{tr} 
\DeclareMathOperator{\ad}{ad} 
\DeclareMathOperator{\End}{End} 
\DeclareMathOperator{\Id}{Id}

\DeclareMathOperator{\Pic}{Pic}

\DeclareMathOperator{\F}{\mathcal{F}}

\newcommand{\f}{\langle f \rangle}

\DeclareMathOperator{\im}{Im}

\DeclareMathOperator{\SU}{SU}              
\DeclareMathOperator{\SL}{SL}
\DeclareMathOperator{\PSL}{PSL}
\DeclareMathOperator{\PU}{PU}
\DeclareMathOperator{\PGL}{PGL}    
\DeclareMathOperator{\Aut}{Aut}

\DeclareMathOperator{\Hom}{Hom}

\DeclareMathOperator{\CS}{CS}

\DeclareMathOperator{\D}{\tilde{\mathcal{D}}}

\DeclareMathOperator{\K}{K}

\DeclareMathOperator{\Flat}{Flat}

\DeclareMathOperator{\rank}{rank}

\DeclareMathOperator{\diag}{diag}
\DeclareMathOperator{\Td}{Td}

\DeclareMathOperator{\Lb}{\mathbbm{L}}
\DeclareMathOperator{\Eb}{\mathbbm{E}}

\DeclareMathOperator{\Jac}{Jac}

\DeclareMathOperator{\ord}{ord}

\DeclareMathOperator{\ch}{ch}
\DeclareMathOperator{\Spin}{Spin}

\DeclareMathOperator{\Parabolic}{Par}

\DeclareMathOperator{\MCG}{MCG}

\DeclareMathOperator{\res}{res}
\newcommand{\X}{\tilde{X}}
\newcommand{\tM}{\tilde{\mathcal{M}}}
\DeclareMathOperator{\E}{\mathcal{E}}

\begin{document}

\title{The Automorphism Equivariant Hitchin Index}

\author{Jørgen Ellegaard Andersen 
\& William Elbæk Mistegård 
}

\maketitle 

\begin{abstract}
	Let $\To$ be the one-dimensional complex torus. We consider the action of an automorphism $f$ of a Riemann surface $X$  on the cohomology of the $\To$-equivariant determinant line  bundle $\LD$ over the moduli space $\M$ of rank two  Higgs bundles on $X$ with fixed determinant of odd degree. We define and study the automorphism equivariant Hitchin index $\chi_{\To}(\M,\LD,f)$. We prove a formula for it in terms of cohomological pairings of canonical $\To$-equivariant classes of certain moduli spaces of parabolic Higgs bundles over the quotient Riemann surface $X/\f$. 
\end{abstract}

\section{Introduction}

Let $X$ be a compact Riemann surface of genus $g\geq 2$ with an automorphism $f \in \Aut(X)$ of odd prime order $p$.  Let $\Lambda\rightarrow X$ be an $\f$-equivariant holomorphic line bundle of odd degree.  Denote by $\M$ the moduli space of stable rank two Higgs bundles on $X$ with fixed  determinant $\Lambda$, as introduced in \cite{hitchinSelfDualityEquationsRiemann1987a,simpsonHiggsBundlesLocal1992}.
 The space $\M$ is a smooth  quasi-projective $\To$-variety of dimension $6(g-1)$. Further, $\M$ supports a $\To$-equivariant universal Higgs bundle $(\Eb,\mathbb{\Phi}) \rightarrow \M\times X.$ We construct in Section \ref{sec:FIXEDLOCUS} a canonical $\To\times \f$-equivariant structure on $(\Eb,\mathbb{\Phi})$. Assume that $X^f\not=\emptyset $ and fix $x_0 \in X^f$. Denote by $\pi_{\M}:\M\times X \rightarrow \M$ the projection. Consider the $\To$-equivariant determinant line bundle $\LD \rightarrow \M$ given by
\begin{equation} \label{def:det} \LD=\det(\Eb_{\mid \M \times \{x_0\}})^{1-g} \otimes \det((\pi_\M)_!(\Eb)[1]),
\end{equation} 
where $(\pi_\M)_!(\Eb)[1]$ is the $1$-shift of the derived pushforward of $\Eb$. The $\To\times \f$-eqivariant structure on $\Eb$ induces one on $\LD^k$ for all $k \in \Z$. Let $t$ be the standard character of $\To$. For $k,q,m,\in \Z,q\geq 0$  denote by $H_{m}^q(\M, \LD^{k})$ the $t^m$-weight sub-space of $H^q(\M,\LD^k)$. These are finite-dimensional $\langle f \rangle$-modules.
\begin{definition} \label{index} The \textit{automorphism equivariant Hitchin index} is the  series
	\begin{equation} \label{eq:index}
	\chi_{\To}(\M,\LD^k,f)(t) =\sum_{m \leq 0} \sum_{q \geq 0} (-1)^q  \tr( f :  H_m^q(\M,\LD^k)\rightarrow H_m^q(\M,\LD^k))\cdot t^m.
	\end{equation}
\end{definition} 

Let  $(h,I,J,K) $ be the complete hyper-Kähler metric on $\M$ introduced in \cite{hitchinSelfDualityEquationsRiemann1987a}, for which $I$ is compatible with the algebraic structure. Then $\LD$ is a quantum bundle  for $(\M,\omega_I)$ in the sense of geometric quantization \cite{kostantQuantizationUnitaryRepresentations1970,souriauStructureSystemesDynamiques1970}. Our study of \eqref{eq:index} is motivated by non-abelian Hodge theory \cite{corletteFlatBundlesCanonical1988,donaldsonTwistedHarmonicMaps1987,hitchinSelfDualityEquationsRiemann1987a}, geometric quantization \cite{kostantQuantizationUnitaryRepresentations1970,souriauStructureSystemesDynamiques1970}, quantum topology \cite{reshetikhinRibbonGraphsTheir1990,reshetikhinInvariants3manifoldsLink1991,turaevQuantumInvariantsKnots1994,wittenQuantumFieldTheory1989} and the goal of generalizing the works \cite{andersenWittenReshetikhinTuraev2013,andersenWittenReshetikhinTuraevInvariantsFinite2012,andersenWittenReshetikhinTuraev2017}. To motivate our results, we will first explain the connection between quantum topology and quantization of moduli spaces. The reader familiar with this can skip to Section \ref{sec:introfixedlocus}. 
\newpage
\paragraph*{Quantum Topology.}  Let $G=\SU(n)$. Chern-Simons theory with gauge group $G$ is a Lagrangian field theory, with action functional defined on connections on principal $G$-bundles on oriented compact three-manifolds. The classical solutions are the flat $G$-connections. It was envisioned by Witten \cite{wittenQuantumFieldTheory1989} that level-$k$ quantum Chern-Simons theory with gauge group $G$ form a three-dimensional TQFT in the sense of Atiyah \cite{atiyahTopologicalQuantumField1988}, and that the evaluation of polynomial knot invariants  \cite{HOMPLY,jonesPolynomialInvariantKnots1985,jonesHeckeAlgebraRepresentations1987} at $\exp(2\pi i /(k+n))$ could be given an intrinsic definition as certain expectation values in quantum Chern-Simons theory.

A three-dimensional TQFT $Z_k$ was subsequently constructed mathematically using surgery presentations of three-manifolds and algebraic and combinatorial means by Reshetikhin and Turaev \cite{reshetikhinInvariants3manifoldsLink1991,reshetikhinRibbonGraphsTheir1990,turaevQuantumInvariantsKnots1994}. This is known as the WRT-TQFT. Let $\Lambda_k$ denote the set of irreducible $G$-representations of level at most $k$. Essentially, the WRT-TQFT is a symmetric monoidal functor from the category of compact oriented three-dimensional cobordisms with oriented framed $\Lambda_k$-labelled tangles, to the category of finite dimensional Hilbert spaces. In particular,  to a compact oriented two-manifold $\Sigma$ with a (possible empty) finite subset of points labelled by $\Lambda_k$, the so-called modular functor \cite{andersenConstructionModularFunctors2017} of the WRT-TQFT assign a finite dimensional Hilbert space $Z_k(\Sigma)$ together with a projective unitary representation of the mapping class group
\begin{equation} \label{WRTQuantumRep}
Z_k: \MCG(\Sigma) \rightarrow \PU(Z_k(\Sigma)).
\end{equation}
The WRT-TQFT gives powerful invariants: The projective representations \eqref{WRTQuantumRep} were proven to be asymptotically faithful by the first author in the work \cite{andersenAsymptoticFaithfulnessQuantum2006}. Further, to a closed oriented three-manifold with a labelled link $(M,L)$ the WRT-TQFT assign a topological invariant $Z_k(M,L) \in \C$. These invariants generalize evaluations of polynomial knot invariants \cite{HOMPLY,jonesPolynomialInvariantKnots1985,jonesHeckeAlgebraRepresentations1987}, are determined by the modular functor through TQFT axioms \cite{atiyahTopologicalQuantumField1988,Kontsevich:1988} and have deep connections to number theory and hyperbolic geometry  \cite{Kashaev}.

In \cite{wittenQuantumFieldTheory1989} Witten argued that the TQFT-representation should be realizable by quantization of moduli spaces of flat $G$-connections. The quantization of moduli spaces construction was achieved independently in \cite{axelrodGeometricQuantizationChernSimons1991} and \cite{hitchinFlatConnectionsGeometric1990} , and this will be reviewed below (in the coprime case). By a large body of work culminating in the work of the first author and Ueno \cite{andersenConstructionWittenReshetikhin2015} and involving conformal field theory \cite{andersenConstructionWittenReshetikhin2015,tsuchiyaConformalFieldTheory1989}, skein theory \cite{blanchetThreemanifoldInvariantsDerived1992,blanchetTopologicalQuantumField1995} and the work \cite{laszloHitchinWZWConnections1998}, it is now established, that the projective representations constructed by quantization \cite{axelrodGeometricQuantizationChernSimons1991,hitchinFlatConnectionsGeometric1990} (and introduced below in equation \eqref{eq:quantumrep}) are equivalent to the WRT-TQFT representations \eqref{WRTQuantumRep}.
\paragraph*{Quantization of Moduli Spaces.}  Let $\Sigma$ be compact oriented two-manifold of genus at least two with a marked point $x_0$ labelled by the element $\lambda \in \Lambda_k$, which corresponds to the conjugacy class $C= [e^{2\pi i/n}\Id]$ under the standard bijection between $\Lambda_k$ and certain conjugacy classes of $G$. Let $\n_{\Flat}$ denote the moduli space of flat principal $G$-bundles on $\Sigma^*=\Sigma \setminus\{x_0\} $ with holonomy  around the puncture contained in $C$. The space $\n_{\Flat}$ is a compact symplectic manifold   equipped with the Atiyah-Bott-Goldman symplectic form \cite{atiyahYangMillsEquationsRiemann1983,goldmanSymplecticNatureFundamental1984} and it is a symplectic leaf of the space of classical solutions in Chern-Simons theory \cite{chernCharacteristicFormsGeometric1974a,freedClassicalChernSimonsTheory1995} on $\Sigma^* \times \R$. Further, it supports a pre-quantum line bundle $\LD_{\CS}\rightarrow \n_{\Flat}$ in the sense of \cite{kostantQuantizationUnitaryRepresentations1970,souriauStructureSystemesDynamiques1970,woodhouseGeometricQuantization1992}. To perform Kähler quantization one needs to introduce a Kähler structure. There exists a natural family of Kähler structures parametrized by Teichmüller space $\mathcal{T}$, as we will now explain. Pick a complex structure $\sigma \in \mathcal{T}$ and let $X_{\sigma}=(\Sigma,\sigma)$ be the associated Riemann surface.  By the Narasimhan-Seshadri theorem \cite{narasimhanStableUnitaryVector1965} $\n_{\Flat}$ is diffeomorphic to the projective variety given by the moduli space $\n_{\sigma}$ of stable holomorphic vector bundles on $X_{\sigma}$ of rank $n$ and fixed determinant $\mathcal{O}_X(-x_0)$. This is a smooth projective variety and $\LD_{\CS} \rightarrow \n_{\Flat}$ is isomorphic to Quillen's \cite{qvillenDeterminantsCauchyRiemannOperators1985} determinant line bundle $\LD_{\sigma} \rightarrow \n_{\sigma}.$ This line bundle is ample and the complex structure on $\n_{\Flat}$ induced via this identification is a Kähler structure with respect to the Atiyah-Bott-Goldman symplectic form. Consider the level-$k$ quantization of $\n_{\Flat}$ with respect to this Kähler polarization 
$ H_{k,\sigma}=H^0(\n_{\sigma},\LD_{\sigma}^k).
$ The family of all such Kähler-quantizations of $\n_{\Flat}$ form a smooth vector bundle $H_k\rightarrow \mathcal{T}$. By the independent works \cite{hitchinFlatConnectionsGeometric1990} and  \cite{axelrodGeometricQuantizationChernSimons1991} this bundle supports a projectively flat $\MCG(\Sigma)$-invariant connection, called a Hitchin connection in \cite{andersenHitchinConnectionToeplitz2012}. Fix $\sigma\in \mathcal{T}$ and write $X=X_{\sigma},\n=\n_{\sigma}$ and $\LD=\LD_{\sigma}$. The monodromy of the Hitchin connection results in a projective representation 
\begin{equation} \label{eq:quantumrep}
 V_k:\MCG(\Sigma) \rightarrow \PGL(H^0(\n,\LD^k)),
\end{equation} 
and this is the aforementioned projective representation, which is projectively equivalent to the WRT-TQFT-representation \eqref{WRTQuantumRep}.  

Consider an automorphism $f \in \Aut(X)$. In this case, the action $V_k(f)$ is given by the action of $f$ on $H^0(\n,\LD^k)$. Consider the three-manifold given by the mapping torus of the automorphism $M_f=\Sigma \times I / \sim.$ Let $K\subset M_f$ be the knot traced out by the marked point. Label $K$ by $\lambda$. Consider the WRT invariant $Z_{k}(M_f,K) \in \C$.  As detailed in the works \cite{andersenWittenReshetikhinTuraev2013,andersenWittenReshetikhinTuraevInvariantsFinite2012,andersenWittenReshetikhinTuraev2017}, this invariant can be computed (up to framing correction) by the algebro-geometric formula
\begin{equation} \label{eq:tr}
Z_{k}(M_f,K)=\tr(f:H^0(\n,\LD^k) \rightarrow H^0(\n,\LD^k))=\chi(\n,\LD^k,f).
\end{equation}
 In the works \cite{andersenWittenReshetikhinTuraev2013,andersenWittenReshetikhinTuraevInvariantsFinite2012,andersenWittenReshetikhinTuraev2017} the index $\chi(\n,\LD^k,f)$ was studied (in more general cases with singular moduli spaces of parabolic vector bundles) by a powerful localization technique valid for projective varieties with singularities, developed in \cite{baumRiemannRochSingularVarieties1975,baumRiemannRochTopologicalKtheory1979}. This was used to prove Witten's asymptotic expansion conjecture \cite{wittenQuantumFieldTheory1989} for $M_f$. For elements $\varphi \in \MCG(\Sigma)$, whose action on $\n_{\Flat}$ satisfy a genericity condition, this conjecture was proven for the mapping torus $M_{\varphi}$ using the representations \eqref{eq:quantumrep} in the work \cite{andersenAsymptoticExpansionsWitten2019}, which generalized a result from \cite{charlesAsymptoticPropertiesQuantum2010}.
\paragraph*{Quantum Topology and Higgs Bundles.} Consider, while specializing to rank two, the index $\chi_{\To}(\M,\LD^k,f)$ introduced in \eqref{eq:index}: it is a generalization of the index \eqref{eq:tr} which equals the WRT-TQFT invariant of the mapping torus, and the study of \eqref{eq:index} is motivated by passing from compact gauge group $G$ to complex gauge group $G_{\C}=\SL(2,\C)$. To explain the latter, recall that by non-abelian Hodge theory \cite{corletteFlatBundlesCanonical1988,donaldsonTwistedHarmonicMaps1987,hitchinSelfDualityEquationsRiemann1987a} we have a diffeomorphism
$\M \cong \M_{\Flat},$ where $\M_{\Flat}$ is the moduli space of flat $G_{\C}$-connections on $\Sigma^*$ with holonomy minus the identity around the puncture.

Our main results are Theorem \ref{thm:M^f}, which identifies the fixed variety $\M^f$ with a union $\tM$ of moduli spaces of parabolic Higgs bundles on the quotient Riemann surface, and Theorem \ref{ThmMain}, which gives a formula for \eqref{eq:index} as a sum of cohomological pairings on $\tM^{\To}$. These results are the basis for ongoing work that aims at explicitly evaluating \eqref{eq:index}. Further, in accordance with the TQFT picture, our Corollary  \ref{Cor:TopInv} asserts that \eqref{eq:index} is determined by the Seifert invariants of the mapping torus. 
We now present these results in full detail.

\subsection{The Fixed Locus} \label{sec:introfixedlocus}
 We use the $\To$-localization result of Wu \cite{wuInstantonComplexHolomorphic2003} and the classical Lefschetz fixed point formula \cite{atiyahIndexEllipticOperators1968} to study \eqref{eq:index}. In connection hereto, we identify the fixed loci $\M^{f}$ and $\M^{\To \times \f}.$ Denote by $\n\subset \M$ the $\f$-subvariety given by the moduli space of stable holomorphic vector bundles with fixed determinant $\Lambda.$ The fixed locus $\n^f$ was identified in \cite{andersenAutomorphismFixedPoints2002}, and we generalize the method from that article. 

For each $x \in X^f$, define the $p$'th root of unity $\alpha_x \in \mu_p$ such that the $f$-action on $\Lambda_x$ is given by $\alpha_x \cdot \Id$. Denote by $\mathcal{I}_0$ the set of maps $i$ from  $X^f$ to the set of unordered pairs of $p$'th roots of unity, such that $m(i(x))=\alpha^{-1}_x$ for all $x \in X^f$, where $m$ is the multiplication map. For all $x\in X^f$ let $\sqrt{\alpha_x}$ be the unique square-root of $\alpha_x$ which belong to $\mu_p$. We can and will identify the set $\mathcal{I}_0$ with the set of all tuples $ (i_x)_{x \in  X^f} $ of non-negative integers  with $i_x  \leq (p-1)/2$ for all $x\in X^f$. Under this identification, the map $i$ corresponding to a tuple $(i_x)_{x \in X^f}$ is determined by $	x \mapsto [(\sqrt{\alpha_x}^{-1}\zeta^{i_x}, \sqrt{\alpha_x}^{-1}\zeta^{-i_x})]$, where $\zeta=e^{2\pi i/p}$. 

Consider the projection $\pi: X \rightarrow \X$, where $\X=X/ \f$ is the quotient Riemann surface. Write $\D=\pi(X^f)$ and note that $\pi: X^f \rightarrow  \D$ is a bijection. For all $x\in X^f$ we will write $\pi(x)=x$. For each $i \in \mathcal{I}_0$, let $\D_i=\{x\in \D: i_x\not=0\}$.  Let $\tM_i$ be the moduli space \cite{konnoConstructionModuliSpace1993}  of rank two stable parabolic Higgs bundles on $(\X,\D_i)$ with fixed determinant $\tilde{\Lambda}_i$ and weights $(w_{i,1}(x),w_{i,2}(x))_{x \in \D_i} \in (\Q^2)^{\D_i}$ both introduced in Definition \ref{def:beta}. This is a smooth $\To$-variety. Let $\mathcal{I}\subset \mathcal{I}_0$ be the subset of $i\in \mathcal{I}_0$ such that $\tM_i$ is non-empty. Identify $\f$ with the group $\mu_p$ and consider $\tM_i$ as a $\mu_p$-variety with the trivial action. Write $\tM=\sqcup_{i \in \mathcal{I}} \tM_i$. In Proposition \ref{pro:uniquelift} we construct an $f$-action on $\Eb$ that covers $(f^*,f^{-1}): \M \times X \rightarrow \M \times X$.  
 \begin{theorem} \label{thm:M^f}
The set of components of $\M^f$ is in bijection with $\mathcal{I}$. Fix $i\in \mathcal{I}$ and write $\M_i$ for the corresponding component. \begin{itemize} \item  The component $\M_{i}$ is a $\To$-subvariety of $\M$ and is characterized by the fact that the eigenvalues of the $f$-action on $\Eb_{\mid \M_{i}\times X^f}$ are given by $i.$ The map from Definition \ref{def:varXi} is an isomorphism of $\To$-varieties $\label{isomM_i}
	\varXi_i :\M_i \rightarrow \tM_i.$ 
	 \item The triple introduced in equation \eqref{eq:uiversalparabolic} is a universal $\To$-equivariant parabolic Higgs bundle $(\tilde{\Eb}_i,\tilde{\mathbb{F}}_i,\tilde{\mathbb{\Phi}}_i) \rightarrow \tM_i \times \X.$ The isomorphism of $\To$-varieties $\varXi_i \times \pi: \M_{i}\times X^f \rightarrow \tM_i\times \mathcal{D}$ is covered by an isomorphism of $\To$-equivariant vector bundles
\begin{equation} \label{mu_pequi}  \Eb_{\mid \M_i \times X^f} \rightarrow \tilde{\Eb}_{i \mid \tM_i \times \D}.
\end{equation}

\end{itemize} 
 \end{theorem} 
As a universal parabolic Higgs bundle is essentially unique if it exists, the point of the second bullet of Theorem \ref{thm:M^f} is the explicit construction of it.

 Theorem \ref{thm:M^f} is of independent interest. A similar identification was obtained in the work \cite{garcia-pradaActionMappingClass2020} in the form of a bijection between $\M^f$ and a moduli space of parabolic Higgs bundles on $\X$. However, they work with a moduli space of parabolic principal  Higgs bundles, which does not admit a $\To$-action, and they don't consider the equivariant universal bundle.

\subsection{Localization}  \label{sec:introlocalization}

 For all $i \in \mathcal{I}$ write $\mathcal{J}_i=\pi_0(\tM_i^{\To})$ and denote the components of $\tM_i^{\To}$ by $\{\tM^{\To}_{i,j}\}_{j \in \mathcal{J}_i}$. These components are determined in Corollary \ref{Cor:B}.  
 Fix $i \in \mathcal{I}$ and $j\in \mathcal{J}_i$. For any algebraig group $H$, denote by $K^0_{H}$ the functor that assigns the $K$-theory ring of $H$-equivariant algebraic vector bundles to a smooth projective $H$-variety. Write $K_{i,j}=K^0_{\To \times \mu_p}(\tM^{\To}_{i,j})$. Below equivariant means equivariant with respect to $\To\times \mu_p$.  Let $\Z[\mu_p]$ denote the representation ring of $\mu_p$. Let $\hat{\zeta}$ be the standard representation of $\mu_p$. We have the following identifications \footnote{We use the standard K-theory convention of denoting the tensor product by the multiplication symbol.}
 \begin{equation} \label{decomp}
K_{i,j} \cong \Z[\mu_p]\otimes_{\Z} K^0_{\To}(\tM^{\To}_{i,j})\cong \bigoplus_{l=0}^{p-1} \hat{\zeta}^l\cdot K^0_{\To}(\tM^{\To}_{i,j})  \cong \bigoplus_{l=0}^{p-1}  \hat{\zeta}^l \cdot K^0(\tM^{\To}_{i,j})[t,t^{-1}].
\end{equation}
  Let $\hat{K}_{i,j}=K_{i,j}[[t^{-1}]].$ Let $S\subset \Z[\mu_p]$ denote the multiplicatively closed subset generated by the elements $1-\hat{\zeta}^l$ where $l$ runs through the set $\{1,...,p-1\}$. Towards stating Theorem \ref{ThmMain} we will first introduce a class $\tilde{\chi}_{i,j} \in S^{-1}\cdot \hat{K}_{i,j}$ (defined in equation \eqref{def:chigamma}) and a certain cohomology class $\tilde{\omega}_i \in H^2_{\To}(\tM_{i})$ (defined in equation \eqref{eq:omega_i}). The next paragraphs introduces the notation necessary to define these classes.
  
  For every $W \in K_{i,j}$, denote by $W^{+}$ (resp. $W^{-}$) the strictly positive (resp. negative) part w.r.t. the $\To$-grading and denote by $W^{\To}$ the constant part w.r.t the $\To$-grading. Denote by $W^{\To\times \mu_p}$ the projection onto the $(\To\times \mu_p)$-invariant part. Denote by $s$ the equivariant operation of taking the sum of symmetric powers, and denote by $\lambda$ the equivariant operation of taking the alternating sum of exterior powers.Then $\lambda((W^{\To}/W^{\To\times \mu_p})^{*})^{-1}$ exists as a class in $S^{-1}\cdot K_{i,j}$ (see Section \ref{sec:Localization} on localization), and  we define the equivariant class
\begin{equation} \label{def:Omega}
\Omega(W)=\det(W^{-})\cdot s(W^{-})\cdot  s((W^+)^*) \cdot \lambda((W^{\To}/W^{\To\times \mu_p})^{*})^{-1} \in S^{-1}\cdot \hat{K}_{i,j}.
\end{equation}
 Consider the universal parabolic Higgs bundle $\tilde{\Eb}_i \rightarrow \tM_i \times \X$ provided by Theorem \ref{thm:M^f}. We equip the restriction of $\tilde{\Eb}_i$ to $\tM_i \times \D$ with the $(\To\times \mu_p)$-equivariant structure induced from \eqref{mu_pequi}. We now present the $\mu_p$-weight subbundle decomposition of the restriction of $\tilde{\mathbbm{U}}_i=\End_0(\tilde{\Eb}_i)$ to $\tM_i \times \D$, which is computed at the end of Section \ref{sec:Hecke}. Let  $x \in \D$. 
For $x  \notin \D_i$, the $\mu_p$-action on $\tilde{\mathbb{U}}_{i,x}$ is trivial.  Assume now that $x \in \D_i$.  Consider the line bundle on $\tM_i$ given by $\Lb_{i,x}=\Hom(\tilde{\mathbb{F}}_{i,x},(\tilde{\Eb}_{i,x}/\tilde{\mathbb{F}}_{i,x}))$ which was introduced in \cite{biswasCanonicalGeneratorsCohomology1996}. We have that
\begin{equation} \label{uni}
\tilde{\mathbb{U}}_{i,x}=\bigoplus_{\epsilon=-1,0,1} \Lb^{\epsilon}_{i,x} \otimes \hat{\zeta}^{-2\epsilon i_x}. 
\end{equation} For all $x\in X^f$ define the integer $n_x \in \{1,...,p-1\}$ by the condition that $df_x= \zeta^{n_x} \cdot \Id$, and write $\tilde{\mathbb{U}}_{i,j,x}$ for the restriction of $ \tilde{\mathbb{U}}_{i,x} $ to $\tM^{\To}_{i,j}$. Define the equivariant class
\begin{equation} \label{def:V_x} 
\tilde{V}_{i,j}= \sum_{x \in \D} \tilde{\mathbb{U}}_{i,j,x}  \cdot(1-t\cdot \hat{\zeta}^{-n_x})  \sum_{l=1}^{p-1}  l \hat{\zeta}^{-l  n_x}.
\end{equation}
 Denote by $T_{i,j} \tM_i$ the restriction of $T \tM_i$ to $\tM^{\To}_{i,j}$. Let $(\tilde{V}_{i,j})^{\mu_p}$ be the projection of $\tilde{V}_{i,j}$ onto the $0$'th summand in \eqref{decomp}. Denote by $\Phi_p$ the $p$'th cyclotomic polynomial. In the proof of Theorem \ref{ThmMain} we argue the existence of an equivariant class $\tilde{\chi}_{i,j} \in S^{-1}\cdot\hat{K}_{i,j}$ given by the following formula \begin{align} \label{def:chigamma}\tilde{\chi}_{i,j}&=\Omega(T_{i,j} \tM_i\cdot \Phi_p(\hat{\zeta})) \cdot \Omega(\tilde{V}_{i,j}-(\tilde{V}_{i,j})^{\mu_p}\cdot\Phi_p(\hat{\zeta}))^{\frac{1}{p}} ,\end{align}
 with the $p$'th root choosen such that the leading term (with respect to the $t^{-1}$-grading) of the equivariant $\rank(\tilde{\chi}_{i,j})$ is equal to the element $\tilde{\zeta}_{i,j}$ introduced in Definition \ref{def:zetagamma}.
 
 Let $[\X]$ be the fundamental class of $\X$ in homology. Define the $\To$-equivariant cohomology classes $\tilde{\alpha}_i=c_2(\tilde{\mathbb{U}}_i) \cap [\X]/2$ and $\tilde{\delta}_{i,x}=c_1(\mathbb{L}_{i,x}), x \in X^f$. Define the equivariant cohomology class 
 \begin{equation} \label{eq:omega_i} \tilde{\omega}_i= p \tilde{\alpha}_i+p\sum_{x \in \D_i} w_{i,2}(x)\tilde{\delta}_{i,x}.
 \end{equation} By definition of the weight function $w_i$ (given in Definition \ref{def:ParabolicStructure}) the class $\tilde{\omega}_i$ is an integral class. Write $\tilde{\omega}_{i,j}$ for the restriction of this class to $\tM_{i,j}$. Denote by $\ch$ the equivariant Chern character. Let $m_i \in \{0,..,p-1\}$ the unique integer such that the $f$-action on $\LD_{\mid \M_i}$ is given by multiplication by $\zeta^{m_i}$. It will be shown in the proof of Theorem \ref{ThmMain} that 
 \begin{equation} \label{eq:identificationofinducedTLibebundle}
 \ch(({\varXi_{i}}^{-1})^{ *}(\LD_{\mid \M_i}))=\hat{\zeta}^{m_i}\exp(\tilde{\omega}_i).
 \end{equation} Let $\tilde{\n}_i \subset \tM_i$ be the moduli space of holomorphic parabolic vector bundles on $(\X,\D_i)$ with the same weights. The restriction of $\tilde{\omega}_i$ to $\tilde{\n}_i$ is equal to $p$ times the cohomology class of the natural symplectic form of $\tilde{\n}_i$ coming from its identification with a moduli space of flat $G$-connections on $\X \setminus \D_i$ with prescribed meridional holonomy around the punctures determined by the weight function $w_i$. This follows from \cite[Remark 6.6]{andersenWittenReshetikhinTuraev2013}. Let $\tilde{\nu}_{i,j} $ be the integer introduced in Definition \ref{def:zetagamma}. We can finally state our main result.
\begin{theorem} The automorphism equivariant Hitchin index localizes to \label{ThmMain}
	\begin{equation} \label{eq:Main}
	\chi_{\To}(\M,\LD^k,f)= \sum_{i\in \mathcal{I}}\sum_{j \in \mathcal{J}_i} (-1)^{\tilde{\nu}_{i,j}} \zeta^{km_{i}}\int_{\tilde{\M}^{\To}_{i,j}}  \ch(\tilde{\chi}_{i,j}  )\cdot\exp(k \tilde{\omega}_{i,j})\cdot \Td \tilde{\M}^{\To}_{i,j}.
	\end{equation}
\end{theorem}
The decomposition of $\tilde{\mathbb{U}}_{i,x}$ given in \eqref{uni} determines an expression for $\tilde{\chi}_{i,j}$ via canonical classes \cite{atiyahYangMillsEquationsRiemann1983,biswasCanonicalGeneratorsCohomology1996}. In a planned follow up article, we will derive from Theorem \ref{ThmMain} an explicit formula for the series \eqref{eq:index}. See Remark \ref{rem:intersectionpairings} below. Further, Theorem \ref{ThmMain} have interesting corollaries connected to the equivariant Hitchin index studied in \cite{andersenVerlindeFormulaHiggs2017a,halpern-leistnerEquivariantVerlindeFormula2016} and to quantum topology. We now present these.


\subsubsection{The Galois Case} Consider the case $X^f =\emptyset$. In this case the projection $X \rightarrow \X$ is a Galois covering of degree-$p$. In this case the line bundle $\LD\rightarrow \M$ is defined slightly differently than in equation \eqref{def:det}, as explained in Section \ref{sec:Galois}. Let $\tilde{g}$ be the genus of $\X$ and assume that $\tilde{g}\geq 2$. Denote by $\tM$ the moduli space of rank two Higgs bundles on $\X$ with fixed determinant $\Lambda/ \f$ (this quotient is a well-defined line bundle as explained below). 
 In this case $\M^f$ is connected and the isomorphism from Theorem \ref{thm:M^f} gives an isomorpshim of $\To$-varieties $\varXi:\M^f \rightarrow\tM.$ Further, in this case the class $\tilde{\omega}_i$ introduced in \eqref{eq:omega_i} is the first Chern class of the $p$'th tensor power of the $\To$-equivariant determinant line bundle $\tilde{\LD}\rightarrow \tM$, constructed in equation \eqref{def:tildedet}  through the same formula as in \eqref{def:det} (but with reference to  the universal Higgs bundle over $\tM\times \X$).  For all $m\in \Z, q \in \Z_{+}$, denote by $H_m^q(\tM,\tilde{\LD^k}) $ the $t^m$-weight sub-space of $H^q(\tM,\tilde{\LD}^k) $ and define the equivariant Hitchin index
\begin{equation} \label{eq:HitchinIndex}
\chi_{\To}(\tM,\tilde{\LD}^k)(t)=\sum_{m \leq 0} \sum_{q\geq 0} (-1)^q  \dim(H_m^q(\tM,\tilde{\LD}^k))\cdot t^m.
\end{equation}
The series \eqref{eq:HitchinIndex} was studied for the moduli stack of Higgs bundles of any rank by the first author, Gukov and Pei in \cite{andersenVerlindeFormulaHiggs2017a} resulting in a beautiful formula in terms of Lie theory, and it was studied for the rank two case by Halpern-Leistner in \cite{halpern-leistnerEquivariantVerlindeFormula2016}. In \cite{andersenVerlindeFormulaHiggs2017a} it was also shown that each of the coefficients of the series \eqref{eq:HitchinIndex} is a $(1+1)$-dimensional TQFT. Moreover, the series \eqref{eq:HitchinIndex} is a refinement of the celebrated Verlinde polynomial, which has been thoroughly studied, see e.g. \cite{jeffreyIntersectionTheoryModuli1998, thaddeusConformalFieldTheory1992,verlindeFusionRulesModular1988,wittenTwoDimensionalGauge1992,zagierCohomologyModuliSpaces1995}.

In this article, we pay special attention to the Galois case, and we prove the following result as a corollary of \ref{ThmMain}.
\begin{corollary} \label{thm:Galois} In the case that $X^f$ is empty and $\tilde{g}\geq 2$, we have that
\begin{equation}	\chi_{\To}(\M,\LD^k,f)(t)= 	\chi_{\To}(\tM,\tilde{\LD}^k)(t^p).
\end{equation} 
\end{corollary}
In words, the automorphism equivariant Hitchin index of a Galois transformation of a degree-$p$ Galois covering of a curve, is equal to the equivariant Hitchin index of the base curve evaluated at $t^p$. 

\subsubsection{Determination in terms of Fixed Point Data and Seifert Invariants of The Mapping Torus} For each $\xi=(\xi_1,\xi_2)\in \mu_p^2$ define $d_{\xi}$ to be the number of $x \in X^f $ such that $\zeta^{-n_x}=\xi_1 $ and $\alpha_x=\xi_2 $. All of these are determined by the $\f$-equivariant structure on $\Lambda$, and we set $d_{\Lambda}=(d_{\xi})_{\xi \in \mu_p^2} $.
\begin{definition} \label{def:fixedpointdata} The fixed point data of $[\Lambda] \in K_{\f}^0(X)$ is the tuple $D_{\Lambda}=(g, d_{\Lambda})$.\end{definition}  
Define the integer $l_0\in \{0,...,p-1\}$ by $\alpha_{x_0}=\zeta^{l_0}$, where $x_0$ is the marked fixed point that enters into the definition of the determinant line bundle $\LD\rightarrow \M$ given in equation \eqref{def:det}. Note $\lvert X^f \rvert=\sum_{\xi} d_{\xi}$. Observe that the genus of $\X$ is determined by $D_{\Lambda}$ via Hurwitz's formula. The cohomological pairings in equation \eqref{eq:Main} are determined by the fixed point data and  $l_0$. Thus we have the following 
\begin{corollary} \label{cor:determinationnumericalfixedpointdata}	The series $\chi_{\To}(\M,\LD^k,f)$ is determined by $D_{\Lambda}$. \end{corollary}

As above, denote by $\Sigma$ the underlying two-manifold of $X$. Consider again the mapping torus $M_f$. This is Seifert fibered \cite{orlikSeifertManifolds1972} over the quotient surface $\tilde{\Sigma}=\Sigma/ \f$. We recall the Seifert invariants of $M_f$ in equation \eqref{eq:SeifertInvariants} below. The knot $K$ traced out by the marked point $x_0$ is one of the exceptional fibers. By the Seifert invariants of $(M_f,K)$ we shall mean the Seifert invariants of $M_f$ together with the information of which of the Seifert invariants correspond to $K$. Assume that $\Lambda=\mathcal{O}_X(-x_0)$ is the $\f$-equivariant line bundle associated with the divisor $-x_0$. As mentioned above, this is the natural choice of fixed  determinant arising from gauge theory via non-abelian Hodge theory \cite{corletteFlatBundlesCanonical1988,donaldsonTwistedHarmonicMaps1987,hitchinSelfDualityEquationsRiemann1987a}. In this case, the Seifert invariants are equivalent to $(D_{\Lambda},l_0)$. Therefore, Corollary \ref{cor:determinationnumericalfixedpointdata} trivially implies the following 
\begin{corollary} \label{Cor:TopInv} The series $\chi_{\To}(\M,\LD^k,f)$ is determined by the Seifert invariants of the pair $(M_f,K)$. \end{corollary} 
Corollary \ref{Cor:TopInv} is well-motivated by TQFT axioms \cite{atiyahTopologicalQuantumField1988} as explained in detail above. Furthermore, we will show that the automorphism equivariant Hitchin index determines the WRT-TQFT invariant of the mapping torus through the identity
\begin{equation} \label{eq:algquantumtop}
\lim_{t \rightarrow \infty} (-1)^{3g-3}\chi_{\To}(\M,\LD^k,f)(t)=\chi(\mathcal{N},\LD^{k+2},f).
\end{equation}
This equation is explained in detail Section \ref{sec:QT}. Equation \eqref{eq:algquantumtop} provides a new link between algebraic geometry and quantum topology, which extends in a natural way the link provided by the realization of the WRT-TQFT mapping class group representations via quantization of moduli spaces of holomorphic vector bundles on Riemann surfaces.

 Finally, we mention that for the class of plumbed three-manifolds, there is the BPS $q$-series invariant $\hat{Z}(M;q)$ \cite{gukovTwovariableSeriesKnot2019,gukovBPSSpectra3manifold2017}, which is an integer power series invariant (depending on a choice of $\Spin^c$-structure), which is a refinement of $Z_k(M)$ and connected to $\SL(2,\C)$-Chern-Simons via resurgence \cite{AM22,gukovResurgenceComplexChernSimons2016}. At the time of writing, it is not clear if and how this is connected to the automorphism equivariant Hitchin index.

\subsection{Organization}

The organization of this paper is as follows. In Section \ref{sec:FIXEDLOCUS} we construct the $\To \times \langle f \rangle$-equivariant structure on $(\Eb,\mathbb{\Phi})\rightarrow \M\times X$ 
and prove Theorem \ref{thm:M^f}. In Section \ref{sec:Localization} we prove Theorem \ref{ThmMain} and Corollary \ref{cor:determinationnumericalfixedpointdata} and Corollary \ref{thm:Galois}. Corollary \ref{Cor:TopInv}  is proven in Section \ref{sec:QT}.

\paragraph*{Acknowledgements.} The authors warmly thank T. Hausel for very valuable discussions and assistance concerning Higgs bundles and their properties. The authors also thank the anonymous referee for insightful comments and suggestions that
improved the quality of this paper.

This work is supported by the grant from the Simons foundation, Simons Collaboration on New Structures in Low-Dimensional Topology grant no. 994320, the  ERC-SyG project, Recursive and Exact New Quantum Theory (ReNewQuantum) with funding from the European Research Council under the European Union‘s Horizon 2020 research and innovation programme, grant agreement no. 810573 and the Carlsberg Foundation grant no. CF20-0431.


\section{The Fixed Locus}
\label{sec:Lift}

\begin{remark}
	Let $Y$ be a compact Riemann surface with a finite subset $\delta$ of marked points. Our conventions and notation on quasi-parabolic Higgs bundles and parabolic Higgs bundles on $(Y,\delta)$ agree with those in \cite{konnoConstructionModuliSpace1993}, except we only consider full flags and specialize to the rank two case. Notice that the real numbers $\{\alpha_j(y)\}_{ j=1,2, y \in \delta}$ introduced in Section $1$ in \cite{konnoConstructionModuliSpace1993} satisfy that $\alpha_j(y)=w_j(y),j=1,2$ for all marked points $y \in \delta $. 
\end{remark}

\begin{definition}Let $D \subset X^f$ be a subset.  Given a quasi-parabolic Higgs bundle $(E,F,\Phi) \rightarrow (X,D)$ define the quasi-parabolic Higgs bundle
$f^*(E,F,\Phi)=(f^*(E),f^*(F),\tilde{f}^*(\Phi))$, where $(f^*(E),f^*(F))\rightarrow (X,D)$ is the pullback quasi-parabolic vector bundle and 
\begin{equation} \label{eq:fHiggsfield}
\tilde{f}^*(\Phi)=(\Id_{\End_0(f^*(E))} \otimes (d f)^*)(f^*(\Phi)).
\end{equation} 
\end{definition} 
Consider the case $D=\emptyset$. The map $(E,\Phi) \mapsto f^*(E,\Phi)$ preserves stability and descends to an automorphism of $\M$ of order $p$. 
\begin{definition} An $\f$-equivariant structure on a quasi-parabolic Higgs bundle $(E,F,\Phi)\rightarrow (X,D)$ is an $\f$-equivariant structure on $E \rightarrow X$, such that the generator $f_E$ of the action on $E$ covering the action of $f$ on $X$ induces an isomorphism of quasi-parabolic Higgs bundles $(E,F,\Phi)\rightarrow f^*(E,F,\Phi),$ which, by abuse of notation, will be denoted by $f_E:(E,F,\Phi)\rightarrow f^*(E,F,\Phi)$.
\end{definition}

\begin{proposition} \label{pro:EquivariantStructureonHiggsBundlesRelLineBundle} Let $(E,\Phi)$ be a Higgs bundle with isomorphism class $\E \in \M^f$. There is a unique $\f$-equivariant structure with generator $f_E:(E,\Phi) \rightarrow f^*(E,\Phi)$ covering the action of $f$ on $X$, such that $\det(E)$ is $\f$-equivariantly isomorphic to $\Lambda.$ Denote by $D_E$ the subset of $x \in X^f$ such that the eigenvalues of $f_E(x)$ are disctint, say $\sqrt{\alpha_x}\zeta^{i_x}$ and $\sqrt{\alpha_x}\zeta^{-i_x}$ with $i_x \in \{1,...,(p-1)/2\}$. For each $x \in D_E$ let $F_x$ be the $\sqrt{\alpha_x}\zeta^{-i_x}$-eigenspace of $f_E(x)$. Define the set of flags $F=\{F_x\}_{x \in D_E}$. The isomorphism class of the $\f$-equivariant quasi-parabolic Higgs bundle $(E,F,\Phi,f_E)$ on $(X,D_E)$ depends only on $\E$.
	\end{proposition} 
\begin{proof} We begin with the existence part. By assumption there exists an isomorphism of Higgs bundles $f_E:(E,\Phi) \rightarrow (f^*(E),\tilde{f}^*(\Phi))$ such that
\begin{equation} \label{conjugation}
(f^c_E\otimes \Id_K)(\Phi)=\tilde{f}^*(\Phi)=(\Id_{\End_0(f^*(E))} \otimes  (df)^*)(f^*(\Phi)),
\end{equation} 
where $f^c_E :\End_0(E)\rightarrow \End_0(f^*(E)) $ is the isomorphism given by conjugation with $f_E$ and where $\tilde{f}^*(\Phi)$ is defined by equation \eqref{eq:fHiggsfield}. Notice that equation \eqref{conjugation} is invariant under $f_E \mapsto \alpha \cdot f_E$ for any $\alpha \in \C^{\times}$. Recall that since $(E,\Phi)$ is stable we have $f_E^p \in \Aut(E,\Phi)\simeq \C^{\times}$. Thus, we can normalize to ensure $f_E^p=\Id$ and because of conjugation invariance, the equation \eqref{conjugation} remain true. Thus $f^*([E,\Phi])=[(E,\Phi)]$ implies that there exists a lift  $f_E$ of $f$ to $(E,\Phi)$. This is unique up to multiplying by an element of $\mu_p$. Now let $h: \det(E) \rightarrow \Lambda$ be an isomorphims (this exists by assumption). Let $f_{\Lambda}$ be the canonical $\f$-equivariant structure on $\Lambda$. Then $f_1:=h \circ \det(f_E) \circ h^{-1}$ and $f_{\Lambda}$ are two lifts that define $\langle f \rangle$-equivariant stuctures on $\Lambda$. Thus $r=f_{\Lambda}^{-1} \circ f_1 \in \Aut(\Lambda) \simeq \C^{\times}$ must satisfy $r^p=1$. Since $f_E \mapsto \alpha \cdot f_E$ has the effect that $r \mapsto \alpha^2 \cdot r$ and since $\gcd(2,p)=1$ we see that  there is a unique normalization of $f_E$ (still defining an equivariant structure) such that $r=1$, which is equivalent to $h$ being an $\langle f \rangle$-equivariant isomorphism.

 For the uniqueness part, notice that if $h': \det(E) \rightarrow \Lambda$ is any other isomorphism, then we must have $h'=c\cdot h$ for some $c \in \C^{\times}$. But then $f_1'=h' \circ f_E \circ (h')^{-1}=f_1,$ so $r'=f_{\Lambda}^{-1} \circ f_1'$ also satisfies $r'=r$. Thus the normalization of $f_E$ for which $h$ is an equivariant isomorphism is the same normalization for which $h'$ is. 
\end{proof} 

Recall the notation $(\mathbb{E},\mathbb{\Phi})$ for the $\To$-equivariant universal Higgs bundle. The construction of the $\To$-equivariant structure is presented in \cite[\S 3]{hauselGeneratorsCohomologyRing2004} and \cite[\S 6.3.2]{hauselVeryStableHiggs2021}. The universal bundle is well-defined up to tensoring with a $\To$-equivariant line bundle pulled back from $\M$. Set $\mathbbm{U}=\End_0(\mathbb{E})$. We have that 
\begin{align}  T \M& 
 \label{eq:tooo}   =  R^1 \pi_{*} ( \mathbbm{U} \overset{\ad(\mathbb{\Phi})}{\longrightarrow} \mathbbm{U}\cdot \pi^*(K) ).
\end{align}

We consider $\M\times X$ a $\To\times\f$-variety, with the $\f$-action induced by the automorphism $F$ given by $F(\mathcal{E}, x) =(f^*(\mathcal{E}),f^{-1}(x))$ for all $(\mathcal{E},x) \in \M\times X$. Consider the natural projection onto the second factor $\pi_X: \M\times X \rightarrow X$.
\begin{proposition} \label{pro:uniquelift} There is a unique $\To \times \langle f \rangle $-equivariant structure on $(\Eb,\mathbb{\Phi})$ such that $\det(\Eb)$ is $\langle f \rangle$-equivariantly isomorphic to $\pi_X^*(\Lambda)$. Let $f_{\Eb}$ denote the generator of the $\langle f \rangle$-action on $\Eb$ covering the action of $F$ on $\M\times X$. Through the isomorphism \eqref{eq:tooo}, the generator $f_{\Eb}$ induces the standard $\To\times \langle f \rangle$-equivariant structure on $T\M$. 
 \end{proposition} 
 
 \begin{proof} 
 	
 Note that via a formula similar to \eqref{eq:fHiggsfield}, we can equip $F^*(\Eb)$ with a pullback universal Higgs field, which we denote by $\tilde{F}^*(\mathbb{\Phi})$.	It follows from a standard argument that $(F^*(\Eb), \tilde{F}^*(\mathbb{\Phi}))$ is a universal Higgs bundle. Since the universal Higgs bundle is unique up to tensoring with the pullback of a line bundle on $\M$, and since $\Aut(f)$ act trivially on $\Pic(\M)$, we see that there exists a $\To$-equivariant isomorphism of universal Higgs bundles
 	\begin{equation}
 	F_{\Eb}: (\Eb,\mathbb{\Phi}) \overset{\sim}{\rightarrow} (F^*(\Eb), \tilde{F}^*(\mathbb{\Phi})).
 	\end{equation}

 	We will argue that $F_{\Eb}$ can be normalized to a new isomorphism which induces a $\To\times \langle f \rangle $-equivariant structure. Clearly $F^{p}=\Id$. It follows that $F_{\Eb}^p \in \Aut_{\To}(\Eb,\mathbb{\Phi})$. It is well-known that $\Aut_{\To}(\Eb,\mathbb{\Phi}) $ is one-dimensional. It follows that $F_{\Eb}^P=\lambda \cdot \Id$ for some $\lambda \in \C^{\times}$. Thus we can replace $F_{\Eb}$ with $f_{\Eb}:=\lambda^{1/p}F_{\Eb}$ to obtain a lift that induces a $\To \times \langle f \rangle $-equivariant structure. The lift can be fixed by demanding that $\det(\Eb)$ is $\langle f \rangle$-equivariantly isomorphic to $\pi_X^*(\Lambda)$.

The fact that $f_{\Eb}$ induces the differential of the standard $\f$-action on $\M$ is a routine verification, which can be done by writing out the definition of the isomorphism \eqref{eq:tooo}.  \end{proof}

\begin{remark} \label{rem:eigenvalues}
     The $f$-action $f_{\Eb}$ on $\Eb$ from Proposition 2.2 covers $(f^*,f^{-1}):\M\times X \rightarrow \M \times X$. Therefore, for every $E \in \M^f$
,
the restriction of $f_{\Eb}$ to $\{E \}\times X$ is a lift $E \rightarrow (f^{-1})^*(E)$ that covers  $f^{-1}:X \rightarrow X$. It follows that if the eigenvalues of
the restriction of $f_{\Eb}$ to $\{E \}\times X^f$
are given by $x \mapsto [(\sqrt{\alpha_x}^{-1}\zeta^{i_x},\sqrt{\alpha_x}^{-1}\zeta^{-i_x})]$ for some $i \in \mathcal{I}_0$, then the eigenvalues of the corresponding
restriction of the lift $E \rightarrow f^*(E)$ guaranteed by Proposition \ref{pro:EquivariantStructureonHiggsBundlesRelLineBundle} are given by $x \mapsto [(\sqrt{\alpha_x}\zeta^{-i_x},\sqrt{\alpha_x}\zeta^{i_x})]$, and for each
$x \in X^f$
and each eigenvalue $\eta$, the $\eta$-eigenspace of $f_{\Eb}$ restricted to $E_x$ is the $\eta^{-1}$-eigenspace of $f_{E}$ restricted to $E_x$.
\end{remark}

\subsection{The Target Moduli Space of Parabolic Higgs Bundles} \label{sec:FIXEDLOCUS}

Let $[\cdot]_p:\Z \rightarrow \{0,...,p-1\}$ denote the remainder modulo $p$.  For each $x \in X^f$, define $m_x \in \{1,...,p-1\}$ by the condition that $m_xn_x=-1 \mod p$, where $n_x$ was defined in the introduction by the condition that $df_x =\zeta^{n_x} \cdot \Id$.

 
For each $x \in X^f$, let $b_x \in \{0,...,p-1\}$ be the unique solution to $\zeta^{2b_xn_x}=\alpha_x^{-1}$. Define the $\f$-invariant divisor $B=\sum_{x \in X^f} b_x$, and let $\mathcal{O}_B$ be the $\f$-equivariant line bundle associated with $B$, with the standard $\f$-equivariant structure as presented in Proposition \ref{lem:uniquelift}. Then the eigenvalues of the $f$-action on $\Lambda(2B)$  over $X^f$ are all equal to unity. It follows that there exists a unique holomorphic line bundle of odd degree $\tilde{\Lambda} \in \Pic(\X) $ such that $\Lambda(2B)$ is $\f$-equivariantly isomorphic to  $\pi^*(\Lambda)$ with the canonical $\f$-equivariant structure on the pullback line bundle. Namely, we can define $\tilde{\Lambda}$ as the quotient $\Lambda(2B) / \f$.
 \begin{definition}  \label{def:beta}   Let $i \in \mathcal{I}_0$.
 	 \begin{itemize} 
 	 	 \item  \label{def:ParabolicStructure} Let $D_{i}$ be the divisor associated with the set $\{x \in \D_{i}: [m_x i_x]_p > (p-1)/2\}.$ 
Define the weights $w_{i}: \D_{i} \rightarrow \Q^2$ by $w_{i,1}(x)=0$ and $w_{i,2}(x)=[2m_xi_x]_p/p$. 
\item \label{def:tM} Denote by $\tM_{i}$ (resp. $\tilde{\n}_{i}$)
the moduli space \cite{konnoConstructionModuliSpace1993} of rank two stable parabolic Higgs bundles (resp. vector bundles)  on $(\X,\D_{i},w_{i})$ with fixed determinant   $\tilde{\Lambda}_{i}=\tilde{\Lambda}(D_{i}).$  
Define $ \tM= \bigsqcup_{i \in \mathcal{I}} \tilde{\M}_{i}.$ 
\end{itemize} 
\end{definition}
 
For each $i \in \mathcal{I}$, the moduli space $\tM_{i}$ is constructed as a hyper-Kähler quotient in \cite{konnoConstructionModuliSpace1993} and proven to be a quasi-projective $\To$-variety in \cite{yokogawaCompactificationModuliParabolic1993,yokogawaInfinitesimalDeformationParabolic1995} by means of GIT \cite{mumfordGeometricInvariantTheory1994}.

\subsection{Hecke Transformations and Descent} \label{sec:Hecke}

	 We recall the following fact about $\f$-equivariant line bundles associated with $\f$-invariant divisors.
	\begin{proposition} \label{lem:uniquelift} Let $ (d_x)_{x \in X^f} $ be a tuple of integers, and let $D$ be the associated $f$-invariant Weil divisor given by $D=\sum_{x \in X^f} d_x\cdot x$. The associated holomorphic line bundle $\mathcal{O}_D\rightarrow X$ admits a canonical $\f$-equivariant structure. Let $f_D$ be the generator of the $\f$-action on $\mathcal{O}_D$ covering the action of $f$ on $X$. For all $x \in X^f$ the generator $f_D$ satisfies the following equation $f_{D}(x)=\zeta^{d_x n_x}\cdot \Id$.
	\end{proposition}

The following is an adaption of Lemma $2.6$ in \cite{andersenAutomorphismFixedPoints2002}, and defines Hecke transformations for $\langle f \rangle$-equivariant quasi-parabolic Higgs bundles.

	\begin{proposition} \label{def:HeckeTransformations} Let $L $ be an $\langle f\rangle $-equivariant line bundle on $X$. Let $(E,F,\Phi)$ be an $\langle f \rangle$-equivariant quasi-parabolic Higgs bundle on $X$ such that $\det(E)$ is $\langle f \rangle $-equivariantly isomorphic to $L$. Fix $x \in X^f$. Assume that $\Phi$ is holomorphic in a neighbourhood of $x$. Assume that the eigenvalues of $f_E(x)$ are $1$ and $\zeta^{b_x}$, where $b_x \in \{1,...,p-1\}$, and assume $F_x$ is the the $\zeta^{b_x}$-eigenspace. Consider the $\f$-equivariant line bundle  $L'=L(-x)$. 
		The following holds. There exists an $\langle f \rangle$-equivariant quasi-parabolic Higgs bundle $(E',F',\Phi')$ with determinant $\langle f \rangle$-equivariantly isomorphic to $L'$, together with a morphism of $\langle f \rangle$-equivariant quasi-parabolic Higgs bundles
		$\iota: (E',F',\Phi') \rightarrow (E,F,\Phi)$ with the following properties. \begin{enumerate} \item The morphism $\iota$  restricts to an isomorphism of $\langle f \rangle$-equivariant quasi-parabolic Higgs bundles over the complement of $x$. \item The image of $\iota(x)$ is equal to the $1$-eigenspace of $f_E(x)$. The kernel of $\iota(x)$ is equal to the $\zeta^{b_x-n_x} $-eigenspace of $f_{E'}(x)$, which is equal to $F'_x.$ The other eigenvalue of $f_{E'}(x)$ is equal to unity.
		\end{enumerate} 
		The following implication holds
	\begin{align} \begin{split} \label{implication}
	&\res(\Phi')_x \not=0 \implies b_x=n_x.
	\end{split} 
	\end{align}
		This pair $(E',\iota)$ can be described via the following short exact sequence of $\langle f \rangle$-equivariant coherent sheaves, where $T_x$ is the sky-scraper sheaf supported at $x$ and  associated with the vector space  $E_x/\ker(f_{E}(x)-\Id_{E}(x))$ 
		\begin{equation} \label{eq:LES} 0 \rightarrow  \mathcal{O}_{E'} \overset{\iota}{\rightarrow} \mathcal{O}_{E} \overset{\lambda}{\rightarrow} 
		T_x  \rightarrow 0.\end{equation} 
		
	\end{proposition}

\begin{definition} With notation as in Proposition \ref{def:HeckeTransformations}, the $\f$-equivariant quasi-parabolic Higgs bundle $(E',F',\Phi')$ is called the Hecke transform of $(E,F,\Phi)$ at $x$, and we introduce the notation \begin{equation} \label{inclusion} 	\Psi_{x}(E,F,\Phi):=(E',F',\Phi'),
\end{equation} 
\end{definition}

We now give the proof of Proposition \ref{def:HeckeTransformations}.
\begin{proof} The following facts are already covered in detail in \cite{andersenAutomorphismFixedPoints2002}. The sequence \eqref{eq:LES} defines an $\f$-equivariant quasi parabolic bundle $(E',F')$ together with a vector bundle morphism $\iota: E' \rightarrow E$, which is an isomorphism of $\f$-equivariant quasi parabolic bundles over $X \setminus \{x\}$. Further, one has an $\f$-equivariant isomorphism $\det(E') \simeq \det(E)(-x)$.
	
	Thus, it remains to construct the Higgs field $\Phi'$ and to prove the assertions about it. The Higgs field $\Phi'$ is constructed on $X \setminus \{x\}$ by conjugating $\Phi$ by $\iota$. Thus it is clear that $(E',F',\Phi')$ is an $\f$-equivariant quasi-parabolic Higgs bundle above $X \setminus \{x\}$, and it remains only to analyze $\Phi'$ near $x$. This is done in two steps. First, we show that $\Phi'$ extends meromorphically to $X$ with a potential pole of order one at $x$, and with residue being nilpotent with respect to the flag $F'_x \subset E'_x$. It immediately follows that $(E',F',\Phi')$ is an $\f$-equivariant quasi parabolic Higgs bundle, because $\Phi'$ is fixed by the $\f$-action on a dense subspace. In the second step, we verify the implication \eqref{implication}.

\paragraph*{First part.}  Let $z$ be a local coordinate for $X$ centered at $x$. Choose a local frame $s=(s_1,s_2)$ of $E$ near $x$ such that with respect to this frame we have that $f_E(0)=\diag(1, \zeta_{b_x})$. The existence of such a frame is proven by elementary means in \cite[Section 2]{andersenAutomorphismFixedPoints2002}. Define $s'=(s_1',s_2')=(s_1,z\cdot s_2) $. We now consider $\Phi'$ near $x$. Let $(s_1^*,s_2^*)$ be the induced frame  of $E^*$ near $x$, such that $s_j^*(s_i)=\delta_{i,j}$ for $i,j \in \{1,2\}$, where $\delta_{i,j}$ is the Kronecker delta. For $i,j=1,2$ let $\phi_{i,j} \in \mathcal{O}_{X,x}$ be the regular function such that $\phi_{i,j} d z=s_i^* (\Phi ( s_j)) $. With respect to the frames $(s_1,s_2)$ and $(s_1',s_2')$ we see that $\iota$ is given by $\diag(1,z)$.
With this notation, we have the following matrix decomposition of $\Phi'$ with respect to the ordered frame $(s_1',s_2')$
\begin{align} \label{eq:conjugatedHiggsField}
 &\iota^{-1}  \circ \Phi \circ \iota  = \begin{pmatrix}
\phi_{1,1} & \phi_{1,2} \cdot  z
\\ \phi_{2,1} \cdot z^{-1} &  \phi_{2,2}
\end{pmatrix}  d z.
\end{align}

It is clear from \eqref{eq:conjugatedHiggsField} that $\Phi'$ extends meromorphically to $X$ with a potential pole of order one at $x$, the residue of which is nilpotent with respect to $E'_x \supset F'_x \supset 0$, since  $\C \cdot s'_2(x)=F'_x.$

\paragraph*{Second part.} Recall that $s_2'(x)$ generates the one-dimensional subspace given by $\ker(\iota(x))=F'_x$. Thus equation \eqref{conjugation} implies that \begin{equation}
\label{residueequation} \res(\Phi')_x=\phi_{2,1}(x)\cdot (s'_2\otimes  (s'_1)^*)(x) \in \Hom(E'_x,F'_x).
\end{equation}  It follows that if $\res(\Phi')_x\not=0$, then we must have that $\phi_{2,1}(x)\not=0$. As  $(E,\Phi)$ is fixed, we see that equation \eqref{conjugation} holds for $\Phi$, and at $x$ this evaluates to
\begin{equation}
\begin{pmatrix}
& \phi_{1,1}(x) & \zeta^{-b_x}\phi_{1,2}(x)
\\&  \zeta^{b_x}\phi_{2,1}(x) & \phi_{2,2}(x)
\end{pmatrix} d z = \zeta^{n_y}\begin{pmatrix}
& \phi_{1,1}(x) & \phi_{1,2}(x)
\\& \phi_{2,1}(x) & \phi_{2,2}(x)
\end{pmatrix} d z.
\end{equation}
By inspecting the left lower entries of the matrices of the above equation, we see that if  $\phi_{2,1}(x)\not=0$, then $b_x=n_x \mod p$. 
 \end{proof}

\begin{remark}\label{rem:F_E'} For future reference, we briefly analyze  the generator $f_{E'}$ of the $\f$-action on $E'$ near $x$. Keep the notation from the proof of Proposition \ref{def:HeckeTransformations}. Write $f_E=(F_{i,j})_{1\leq i,j \leq 2}$ as a matrix of regular functions $F_{i,j} \in \mathcal{O}_{X,x}$ with respect to the frame $(s_1,s_2)$ and the pullback frame $(f^*(s_1), f^*(s_2))$. Then  we have the following equation for $f_{E'}$ with respect to the frames $(s'_1,s_2')$ and $(f^*(s_1'),f^*(s_2'))$ 
\begin{align}
\begin{split}  \label{F'conj} f_{E'} &=\begin{pmatrix}
F_{1,1} & F_{1,2} \cdot  z
\\ F_{2,1} \cdot z^{-1} \zeta^{-n_x} &  F_{2,2}\cdot \zeta^{-n_x}
\end{pmatrix}.
\end{split} 
\end{align} 
Here we used $f^*(\iota)= \diag(1,z\cdot \zeta^{n_x})$ with respect to the pair of ordered frames given by $(f^*(s_1),f^*(s_2))$ and $(f^*(s'_1),f^*(s'_2)).$ Since $F_{2,1}$ vanishes at $x$, we see from equation \eqref{F'conj}  that $f_{E'}$ extends holomorphically to $X$. 
\end{remark}

The Hecke transformation $\Psi_x$ at $x \in X^f$  admits an inverse. This fact follows from an adaption of Lemma $2.8$ in \cite{andersenAutomorphismFixedPoints2002}. The following lemma will be useful for analysing stability conditions below. 
\begin{lemma} \label{lem:HeckeLinebundles}
	Let $(E,F,\Phi)$ and $x \in X^f$ be as in Proposition \ref{def:HeckeTransformations} and let $(E',F',\Phi')$ denote the Hecke transform. There is a bijection between the holomorphic line subbundles of $E$ and holomorphic line subbundles of $E'$, which by abuse of notation is denoted by $\Psi_x $, and this has the following properties.
	\begin{enumerate}
		\item  It inducess a bijection between $\langle f \rangle$-equivariant holomorphic line subbundles of $E$ and $\langle f \rangle$-equivariant holomorphic line subbundles of $E'$,
		\item It induces a bijection between $\Phi$-invariant  holomorphic line subbundles of $E$ and $\Phi'$-invariant holomorphic line subbundles of $E'$.
	\end{enumerate} 
	For a line subbundle $H \subset E$ we define $\Psi_x(H)$ to be equal to $H(-x) $ if $H_x \not= \im(\iota(x))$, otherwise we define $\Psi_x(H)$ to be equal to $H$. 
\end{lemma}

For an $\f$-equivariant quasi-parabolic Higgs bundle $\E \rightarrow X$ on which $\f$ act as the identity above $X^f$, one can define a quasi-parabolic Higgs bundle $\E/ \f \rightarrow \X$ by taking the quotient. We say $\E / \f$ is obtained by descent. This is discussed in detail in the proof of Theorem \ref{thm:M^f} below.

For all $i \in \mathcal{I}_0$ define $\M_{i} \subset \M^f$ to be the  $\To$-subvariety of $\M^f$ characterized by the fact that the eigenvalues of the $f$-action on $\Eb_{\mid \M_{i}\times X^f}$ are given by $i.$ It will follow from the proof of Theorem \ref{ThmMain} that $\M_{i} $ is non-empty if and only if $\tM_i$ is non-empty, i.e. if and only if $i \in \mathcal{I}$.

\begin{definition} \label{def:varXi} Let $i \in \mathcal{I}_0$. Define $\mathcal{D}_i=\pi^{-1}(\D_i) $ and define the $\f$-invariant divisor  $S_i$ on $X$ by \begin{equation} S_{i}=\sum \limits_{x \in \mathcal{D}_i} [i_x m_x]_p \cdot x+B.
	\end{equation} For all $x \in X^f$, define the positive integer $b_{i}(x):=[2m_xi_x]_p.$ Define the $\To$-equivariant morphism $\varXi_{i}: \M_{i} \rightarrow  \tM_{i}$ by 
	\begin{equation}
	\varXi_{i}(\E)=\left( \prod \limits_{x \in \mathcal{D}_i}  \Psi_{x}^{b_{i}(x)} (\E (S_{i})) \right)/ \f.
	\end{equation} 
\end{definition}

It is easily seen that $\varXi_{i}(\M_{i}\cap \n)=\tilde{\n}_{i}$, in accordance with the identification of the $\f$-fixed locus of the moduli space $\n$ of stable bundles given in \cite{andersenAutomorphismFixedPoints2002}, though we have used a slightly different isomorphism here.We remark that a similar analysis involving the effect of Hecke transformations was considered by Witten in \cite{Witten18}, though not in the $\f$-equivariant setting.

\begin{proof}[Proof of Theorem \ref{thm:M^f}] For all $i\in \mathcal{I}_0$ we define 
	\begin{equation} \label{Psi_i}
\Psi_{i}=\prod_{x \in \mathcal{D}_i}  \Psi_{x}^{b_i(x)}.
	\end{equation} The proof is divided into three parts. In the first part, we analyze the effect of $\Psi_{i}$ without regard to stability. In the second part, we analyze descent without regard to stability. Finally, in the third part, we complete the proof by taking stability conditions into account. 
	
	 \paragraph*{First part.} Let $i \in \mathcal{I}_0$. We introduce the following notations. Denote by $M_{i}$ the set of isomorphism classes of $\langle f \rangle$-equivariant quasi parabolic rank two Higgs bundles $(E,F,\Phi)  \rightarrow (X,\mathcal{D}_i)$ with determinant $\langle f \rangle$-equivariantly isomorphic to $\Lambda,$ which satisfies the following conditions. For all $x \in X^f$, the eigenvalues of $f_E(x)$ are $\sqrt{\alpha_x}\zeta^{i_x}$ and $\sqrt{\alpha_x}\zeta^{-i_x}$, and $F_x \subset E_x$ is equal to the $\sqrt{\alpha_x}\zeta^{-i_x}$-eigenspace.  Denote by $M_{i}' $ the set of isomorphism classes of $\langle f \rangle$-equivariant quasi parabolic rank two Higgs bundles $(E,F,\Phi)\rightarrow (X,\mathcal{D}_i)$ with determinant $\langle f\rangle$-equivariantly isomorphic to $\pi^*(\tilde{\Lambda}_{i})$ and such that $f_E$ acts as the identity on the fibres of $E$ above $X^f$.

 		 We will now show that the map given in equation \eqref{Psi_i} gives  a bijection $\Psi_{i}: M_{i} \rightarrow M_{i}'$.
 		  First of all, we note that the map $\Psi_{i}$ is well-defined since Hecke transformations at different points $x \in \mathcal{D}_i$ commutes.

	Let $(E_0,F_0, \Phi) \in M_{i}$. Define the flags $F_0(S_i)=((F_0)_x\otimes(S_i)_x)_{x \in X^f}$. Then $(E_1,F_1,\Phi):=(E_0 (S_{i}),F_0(S_{i}),\Phi)$ canonically defines a $\langle f \rangle$-equivariant quasi parabolic Higgs bundle as in Proposition \ref{def:HeckeTransformations} with eigenvalues $\{1,\zeta^{-2i_x}\}_{x \in \mathcal{D}_i}$. Applying at each $x\in \mathcal{D}_i$ a total of $b_{i}(x)$ Hecke modifications to $(E_1,F_1,\Phi_1)$, we obtain a new $\langle f \rangle$-equivariant quasi-parabolic Higgs bundle  $(E_2,F_2,\Phi_2)=\prod_{x \in \mathcal{D}_i}  \Psi_{x}^{b_{i}(x)} (E_{1},\Phi_1)$. By construction, the $\langle f \rangle$-action is trivial on  $E_{\mid X^f}$. This fact follows from Proposition \ref{def:HeckeTransformations}.

		We now show $\det(E_2)\simeq \pi^*(\Lambda_{i})$. As each Hecke transform at $x$ tensors the determinant by $(-x)$, we see that the determinant of $E_2$ is given by
	\begin{align} \begin{split} \label{detE21}&\det(E_2)= \det(E_1)  \left(-\sum_{x \in \mathcal{D}_i}b_{i}(x)\cdot x \right) =
	\Lambda\left( 2S_{i} -\sum_{x \in \mathcal{D}_i}b_{i}(x)\cdot x \right)  
	\\&= \Lambda \left( 2B +\sum_{x \in \mathcal{D}_i} (2[m_xi_x]_p-[2m_xi_x]_p)\cdot x\right).
	\end{split} 
	\end{align}
Recall that $	D_{i}=\{x \in \D_{i},  [m_xi_x]_p>(p-1)/2\}$. If we have $[m_xi_x]_p\leq (p-1)/2$ then $x\notin D_{i}$ and we have $2[m_xi_x]_p=[2m_xi_x]_p$. If not, then $\tilde{x}\in D_{i}$ and one can use that $p$ is odd to show that  $[2m_xi_x]_p=2[m_xi_x]_p-p.$ Thus
	\begin{align} \begin{split}  \label{eq:Dbeta}
 &\sum_{x \in \mathcal{D}_i} (2[m_xi_x]_p-[2m_xi_x]_p)\cdot x=\sum_{x\in \mathcal{D}_i: [m_xi_x]_p>(p-1)/2} p\cdot x=\pi^*(D_{i}).\end{split}
	\end{align}
	Combining equation \eqref{detE21} with equation \eqref{eq:Dbeta}, and using $\pi^*(\tilde{\Lambda})=\Lambda (2B)$, we obtain the desired isomorphism $	\det(E_2) \simeq  \pi^*(\tilde{\Lambda}_{i}).$

	Thus it only remains to verify that $\Phi_2$ is holomorphic on $X \setminus \mathcal{D}_i$ with a pole of order at most one at each $x\in \mathcal{D}_i$ and with residue nilpotent with respect to the flag. For each tuple $a=(a_x)_{x \in \mathcal{D}_i}$ with $0\leq a_x <b_{i}(x)$ define
		
		\begin{equation} \E_a=\prod_{x \in \mathcal{D}_i} \Psi_x^{a_x}(E_1,F_1,\Phi_1).
		\end{equation} It follows from the implication \eqref{implication}, that the Higgs field of $\E_a$ is in fact holomorphic on $X$. Thus, it is only after we apply the last product of Hecke modifications that the resulting Higgs field $\Phi_2$ aquires a potential pole divisor, which will be a subset of $\mathcal{D}_i$ by construction, and each pole will be of order at most one, with residue nilpotent with respect to the flag (this last fact was shown above and follows from \eqref{eq:conjugatedHiggsField}). Thus we have shown that $\Psi_{i}$ defines a bijection $M_{i}\rightarrow M_{i}'$. 
		
 \paragraph*{Second part.}  Descent for holomorphic vector bundles is discussed in detail in \cite{andersenAutomorphismFixedPoints2002}. We now extend descent to Higgs bundles. 	Denote by $\tilde{M}_{i}$ the set of isomorphism classes of quasi-parabolic rank two Higgs bundles on $(\X,\D_{i})$ with fixed determinant $\tilde{\Lambda}_{i}$. We will show that pullback with respect to the quotient map $\pi: X \rightarrow \X$ induces a bijection $\pi^*:\tilde{M}_{i} \rightarrow  M_{i}'$ with inverse given by descent $\E \mapsto \E/\f$.  
	
	 Define the $\f$-invariant divisor $R_{i}=\mathcal{D}_i+(1-p)(X^f \setminus \mathcal{D}_i)$. First we observe that there is a canonical $\f$-equivariant isomorphism
	\begin{equation} \label{isodisso}
	K_X(R_{i}) \simeq \pi^*(K_{\X}(\D_{i})).
	\end{equation} Let $\E'=(E',F',\Phi') \in M_{i}'$. Clearly the quotient $(E'/\f,F'/\f) \rightarrow \X$ is a well-defined quasi parabolic vector bundle. We will show that \begin{equation}\label{phi'Rbeta}
	\Phi' \in H^0(X, \End_0(E')\otimes K_X(R_{i}))^{\langle f \rangle},
	\end{equation} 
	and by composing with the canonical isomorphism \eqref{isodisso}, we can define a holomorphic quotient section  $$\Phi'/\f \in  H^0(\X,\End_0(E'/\f) \otimes  K_{\X}(\D_{i})).$$
	Thus the triple $(E'/\f,F'/\f,\Phi/\f)$ defines a quasi parabolic Higgs bundle on $\X$, which we denote by $ \E' /\f \in \tilde{M}_{i}.$
	
	We now show \eqref{phi'Rbeta}. By the first part of the proof, we see that this is equivalent to proving that for all $\E=(E,F,\Phi) \in M_{i}$, we have that
		\begin{equation} \label{eq:paj2}
		\Phi \in H^0(X,\End_0(E)\otimes K_X((1-p)(X^f \setminus \mathcal{D}_i))).
		\end{equation} Towards the end of proving  \eqref{eq:paj2}, let $x\in X^f \setminus \mathcal{D}_i$. Let $z$ be a coordinate centered at $x$, and let $s=(s_1,s_2)$ be a local frame of $E$ near $x$. For $i,j \in \{1,2\}$, let $\phi_{i,j} \in \mathcal{O}_{X,x}$ be the uniquely determined holomorphic functions defined near $x$, such that with respect to the frame $s$, we have that $\Phi=\phi d z$, where $\phi=(\phi_{i,j})_{1 \leq i,j \leq 2}.$ We must show that
		\begin{equation}
		\min(\ord_x(\phi_{i,j}) \mid i,j\in \{1,2\}) \geq p-1.
		\end{equation}
		As $(E,\Phi)$ is $\f$-invariant, equation \eqref{conjugation} holds (with $f_E$-denoting the canonical $\langle f \rangle$-generator). Recall that $f_E^c$ denotes conjugation by $f_E$. Consider the left hand side of \eqref{conjugation}. Let $f^{i,j}_{l,k} \in \mathcal{O}_{X,x}$ be the holomorphic functions defined near $x$ such that
		\begin{equation} \label{hy}
		f^c_E(\Phi)= (\sum_{l,k} f^{i,j}_{l,k}\cdot \phi_{l,k})_{i,j}  d z.
		\end{equation}
		We now consider the right hand side of \eqref{conjugation}. We have that	\begin{align} \begin{split} \label{yh}
		&	(\Id \otimes (d f)^*)(f^*(\phi_{i,j} d z))=f^*(\phi_{i,j})  (d f)^*(f^*dz) 
		\\&=f^*(\phi_{i,j}) \frac{\partial f}{\partial z} d z=	\zeta^{n_x}f^*(\phi_{i,j})  d z. \end{split}
		\end{align}
		By \eqref{conjugation} we have equality between \eqref{hy} and \eqref{yh}, and this gives
		\begin{equation} \label{oopp}
		\zeta^{n_x}\phi_{i,j}\circ f=\sum_{l,k} f^{i,j}_{l,k}\cdot \phi_{l,k}.
		\end{equation}
		Evaluating at $z=0$ and using that $f_{l,k}^{i,j}(0)=\delta_{l,k}^{i,j},$ (where the right hand side is Kronecker's delta function) we obtain that $\phi_{i,j}(0)=0$ Assume inductively that we have shown that $\phi^{(m)}_{l,k}(0)=0$ for all $(l,k)$ and all $m=0,1,2,...,n-1$ with $n \leq p-2$. Differentiating equation \eqref{oopp} $n$ times and evaluating at $z=0$ and using that $\frac{\partial f }{\partial z}(z)  = \zeta^{n_x}$, we get 
		\begin{align}
		\zeta^{n_x}	\frac{\partial^{n}  \phi_{i,j} \circ f }{\partial z^n} (0)&=\zeta^{n_x(n+1)} \phi^{(n)}_{i,j}(0) =\frac{\partial^{n} }{\partial z^n} \sum_{l,k} f^{i,j}_{l,k}\cdot \phi_{l,k} (0)
		\\&= \sum_{m=0}^{n} \binom{n}{m}  (f^{i,j}_{l,k})^{(n-m)} (0)\phi^{(m)}_{l,k}(0)= \phi^{(n)}_{l,k}(0),
		\end{align}
		where for the last equality, we used our induction hypothesis and that $f_{l,k}^{i,j}(0)=\delta_{l,k}^{i,j}$. In particular $\phi^{(n)}_{l,k}(0)=\zeta^{n_x(1+n)}\phi^{(n)}_{l,k}(0).$ As $n+1\leq p-1$ we have $\zeta^{n_x(n+1)}\neq 1$ and this shows $\phi^{(n)}_{l,k}(0)=0.$

	\paragraph*{Third part.} By Proposition \ref{pro:EquivariantStructureonHiggsBundlesRelLineBundle} we have a natural inclusion $\M_i \subset M_i$.  We now prove that the composition given by  $(\pi^*)^{-1}\circ \Psi_{i}: M_{i} \rightarrow \tilde{M}_{i}$ restricts to an isomorphism of $\To$-varieties $\varXi_{i}: \M_{i} \rightarrow \tM_{i}$.
		Observe that $\varXi_{i}$ is equviariant with respect to the actions of $\To$ on source and target.  Therefore it remains to show that
		\begin{enumerate}
			\item We have $\im(\varXi_{i})=\tM_{i}$.
			\item The set-theoretic maps $\varXi_{i}$ and $\varXi_{i}^{-1}$ are holomorphic. 
		\end{enumerate}
		We start by establishing $\im(\varXi_{i})=\tM_{i}$. Let $\E \in M_{i}.$ Because of the first part of the proof, it is enough to show that the underlying Higgs bundle is stable if and only if $\varXi_{i}(\E)$ is parabolically stable with respect to the parabolic weight $w_{i}: \mathcal{D}_i \rightarrow \Q^2$. Equivalently, we must show that $\E'=\E(S_{i})$ is stable if and only if $\varXi_{i}(\E)$ is parabolically stable with respect to the parabolic weight $w_{i}: \D_{i} \rightarrow \Q^2$. Towards that end, we will first use Lemma \ref{lem:HeckeLinebundles} to show that if $H' \subset \E'$ is any subbundle of rank one which is preserved by the Higgs field and the $\langle f \rangle$-action, then we have (with $H=H'(- S_{i}))$
		\begin{equation} \label{eq:sloperelations}
		\frac{\mu(\E)-\mu(H)}{p}=\frac{\mu(\E')-\mu(H')}{p}= \mu_{\Parabolic}(\varXi_{i}(\E))-\mu_{\Parabolic}(\varXi_{i}(H)),
		\end{equation}
		where we set $$\varXi_{i}(H)=\prod_{x \in \mathcal{D}_i}\Psi_x^{[2m_xi_x]_p}(H(S_{i}))/\f=\prod_{x \in \mathcal{D}_i}\Psi_x^{[2m_xi_x]_p}(H')/\f,$$
		 with notation as in Lemma \ref{lem:HeckeLinebundles}, where $\Psi_x$ was defined for line subbundles. The first equality in equation \eqref{eq:sloperelations} is trivial and the second is  the analog of Lemma $3.2$ in \cite{andersenAutomorphismFixedPoints2002}. We now prove the second equality in \eqref{eq:sloperelations}. As $w_2(x)=[2m_xi_x]_p/p$ forall $x \in \mathcal{D}_i$, we get from the definition of parabolic slope \cite{konnoConstructionModuliSpace1993} that 
		\begin{align}
		& \label{qq1} \mu_{\Parabolic}(\varXi_{i}(\E))= \frac{\deg(\varXi_{i}(\E))+\sum_{x \in \mathcal{D}_i} [2m_xi_x]_p/p}{2},
		\\& \label{qq2} \mu_{\Parabolic}(\varXi_{i}(H))=\deg(\varXi_{i}(H))+\sum_{x \in \mathcal{D}_i: H'_x=F'_x} \frac{[2m_xi_x]_p}{p}.
		\end{align} 
		Since $H'$ is $\f$-equivariant, we see that for all $x\in \mathcal{D}_i$, we have that if $H'_x \not= \im(\iota(x))$ with notation as in Lemma \ref{lem:HeckeLinebundles}, then we must have $H'_x=F'_x.$ Therefore it follows from Lemma \ref{lem:HeckeLinebundles} that 
		$$
		\Psi_x^{[2m_xi_x]_p}(H')=H' \left(\sum \limits_{x\in \mathcal{D}_i: H'_x=F'_x} -[2m_xi_x]_p \cdot x\right),
		$$
		 and therefore
		\begin{align}\begin{split} \label{qqq1}
		&p\deg(\varXi_{i}(H))
	=\deg(H')-\sum_{x \in \mathcal{D}_i: H_x'= F'_x} [2m_xi_x]_p.
		\end{split} 
		\end{align}
		We have
		\begin{equation} \label{qqq2}
		p\deg(\varXi_{i}(\E))=\deg(E')-\sum_{x \in \mathcal{D}_i} [2m_xi_x]_p.
		\end{equation}
		Substituting equation \eqref{qqq1} into equation \eqref{qq1} and  equation \eqref{qqq1} into equation \eqref{qq2} and then computing the difference  $\mu_{\Parabolic}(\varXi_{i}(\E))-\mu_{\Parabolic}(\varXi_{i}(H))$ gives
		\begin{align}
		&\mu_{\Parabolic}(\varXi_{i}(\E))-\mu_{\Parabolic}(\varXi_{i}(H))
	=\frac{\deg(\varXi_{i}(\E))+\sum_{x \in \mathcal{D}_i} [2m_xi_x]_p/p}{2}
		\\&-\left(\deg(\varXi_{i}(H))+\sum_{x \in \mathcal{D}_i: H'_x=F'_x} \frac{[2m_xi_x]_p}{p} \right) 
		\\&=\frac{\deg(E')-\sum_{x \in \mathcal{D}_i} [2m_xi_x]_p+\sum_{x \in \mathcal{D}_i} [2m_xi_x]_p}{2p}
		\\&-\left(\frac{\deg(H')-\sum \limits_{x \in \mathcal{D}_i: H_x'= F'_x} [2m_xi_x]_p}{p}+\sum \limits_{x \in \mathcal{D}_i: H'_x=F'_x} \frac{[2m_xi_x]_p}{p} \right)
		\\&=\frac{\mu(\E)-\mu(H)}{p}.
		\end{align}
		Thus we have proven equation \eqref{eq:sloperelations}. We now use this to prove $\im(\varXi_{i})=\tM_{i}$.  Observe that for any $(E,F,\Phi)\in M_{i}'$, there is an induced bijection between $\Phi$-invariant and $\langle f \rangle$-equivariant  line subbundles $H$ of $E$ and $\Phi/\f$-invariant line subbundles of $E/\f$, and this is given by $H\mapsto H/\f$. This fact and equation \eqref{eq:sloperelations} implies $\im(\varXi_{i}) \subset \tM_{i}.$ Now let  $\E=(E,\Phi) \in M_{i}$ and assume $\varXi_{i}(\E) \in \tM_{i}$. Assume towards a contradiction that $\E$ is not stable. Since $\deg(\Lambda)$ is odd, this is equivalent to $\E$ not being semi-stable. Consider the Higgs Harder-Narasimhan filtration $0 \subsetneq H \subsetneq E$.  Observe that $f_E(f^*(H))$ is a line subbundle which also induce a Higgs Harder-Narasimhan filtration, and therefore we must have $f_E(f^*(H))=H$ by uniqueness. Therefore $H$ is $\f$-equivariant and $\Phi$-invariant. Hence we can apply equation \eqref{eq:sloperelations} to obtain a contradiction. Thus $\E$ is stable. 
		
		We now argue that $\varXi_{i}$ and its inverse are morphisms of varieties. Both the target and the source of $\varXi_{i}$ can be obtained as GIT quotients \cite{simpsonHiggsBundlesLocal1992,yokogawaCompactificationModuliParabolic1993,yokogawaInfinitesimalDeformationParabolic1995}, and therefore it suffices to show that one can apply Hecke transformations in families. This was done for stable bundles in \cite{andersenAutomorphismFixedPoints2002}, and the analysis therein can readily be repeated for Higgs bundles. In particular, we can apply Hecke modifications to $(\Eb, \mathbb{\Phi})$ and we define
		\begin{equation} \label{eq:uiversalparabolic}
		(\tilde{\Eb}_i,\tilde{\mathbb{F}}_i, \mathbb{\Phi}_i)=\Psi_i((\Eb, \mathbb{\Phi})_{\mid \M_i \times X}) / \f.
		\end{equation}
		It is a routine verification that $(\tilde{\Eb}_i,\tilde{\mathbb{F}}_i, \mathbb{\Phi}_i) \rightarrow \tM_i \times \X$ is a $\To$-equivariant universal parabolic Higgs bundle, and that we have a canonical equivariant isomorphism as in equation \eqref{mu_pequi}. 
\end{proof}
To see that \eqref{uni} holds we argue as follows. For each $E \in \M_i$ and $x \in \mathcal{D}_i$
the flag $F_x \subset E_x$ as specified by Proposition \ref{pro:EquivariantStructureonHiggsBundlesRelLineBundle} is given
by the $\sqrt{\alpha_x} \zeta^{-i_x}$-eigenspace of the unique lift $f_E$ guaranteed by Proposition \ref{pro:EquivariantStructureonHiggsBundlesRelLineBundle}. It follows therefore from Remark \ref{rem:eigenvalues} that $f_{\Eb}$
act with weight $\sqrt{\alpha_x}^{-1} \zeta^{i_x}$ on $\tilde{\mathbb{F}}_{i,x}$. As $\Lb_{i,x}=\Hom(\tilde{\mathbb{F}}_{i,x},(\tilde{\Eb}_{i,x}/\tilde{\mathbb{F}}_{i,x})) \cong \tilde{\mathbb{F}}_{i,x}^*\otimes (\tilde{\Eb}_{i,x}/\tilde{\mathbb{F}}_{i,x})$ this implies that \eqref{uni} holds.

\subsection{Identification of The Torus Fixed Locus} \label{sec:M^fcapM^T}

\begin{definition} Let $C$ denote the set of odd integers $c$ with $1\leq c\leq 2g-3$. For each $c \in C$, let $\F_c$ be the set of isomorphism classes of rank two Higgs bundles $(E,\Phi_b)$ on $X$ of the following form.
	\begin{itemize}
		\item There are holomorphic line bundles $L_1$ and $L_2$ giving a direct sum decomposition $E=L_1 \oplus L_2$, and such that $L_1  L_2 \simeq \Lambda$.
		\item The Higgs field $\Phi_b$ is the off-diagonal Higgs field associated with a section $b \in H^0(X,L_2 L_1^{-1} K_X)$. The degree of the zero divisor of $b$ is equal to $c$.
	\end{itemize}
	\end{definition} 

For $c\in \Z_{\geq 0}$ let $X_c$ (resp. $\X_c$) denote the space of positive effectice divisors of degree $c$ on $X$ (resp. $\X$). Now fix $c \in C$. The following facts are shown by Hitchin in  \cite{hitchinSelfDualityEquationsRiemann1987a}. Every Higgs bundle in $\F_c$ is stable, and thus $\F_c$ is a naturally a subspace of $\M$. Each $\F_c$ is connected and we have 
\begin{equation} \M^{\To}= \n \sqcup_{c \in C} \F_c.
\end{equation} 
Moreover, the map $\delta:\F_c \rightarrow X_c$, given by $(E, \Phi_b) \mapsto (b)  \in X_c$, gives $\F_c$ the structure of a principal $\Jac(X)[2]$-bundle (with action given by tensoring). We notice that pullback with respect to $f$ naturally equips $\F_c, \ \Jac(X)[2]$ and $X_c$ with an $\f$-action.

Fix $c \in C$. Recall \cite[proof of Lemma 6.1]{hauselRelationsCohomologyRing2003} that the universal Higgs bundle restricts to a direct sum of $\To$-equivariant line bundles $\mathbb{E}_{c,1}\oplus \mathbb{E}_{c,0} \rightarrow \F_c\times X$, where $\To$ acts with weight $t$ on $\mathbb{E}_{c,1}$ and with weight equal to unity on $\mathbb{E}_{c,0}$. Fix $i \in \mathcal{I}$. We will now identity the components of $\M_i \cap \F_c$ as well as the $\To\times \f$-weight subbundle decomposition of $\Eb$ restricted to these components. 
 
 \begin{definition}
 	 	Set $\mathcal{J}_c=\sqcup_{i \in \mathcal{I}}\mathcal{J}_{i,c}$, where for each $i\in \mathcal{I}$, we define $\mathcal{J}_{i,c}$ to be the set of divisors of the form $j=\sum_{x \in X^f} j_x \cdot x,$ such that $l_j=(c-\deg(j))/p$ is a non-negative integer and for all $x \in X^f$, we have that $j_x \in \{0,...,p-1\}$ and
 	\begin{align} 
 & \label{ix2}  (2i_x)^2 = (n_x(1+j_x))^2 \mod p.
 	\end{align} 
  
 \end{definition}
 Let $i \in \mathcal{I} ,c \in C$ and $j \in \mathcal{J}_{i,c}$. For all $x \in X^f$, it follows from equation \eqref{ix2} that there is a unique $\epsilon_{j,x} \in \{1,-1\}$ such that 
 \begin{equation} \label{eq:epsilon}
 \epsilon_{j,x}2i_x=-n_x(1+j_x)  \mod p.
 \end{equation}

\begin{lemma} \label{lem:S^m(C)^f}   The set of components of the $\f$-fixed locus $X_c^f$ is in bijection with $\mathcal{J}_{c}$,  and the set of components of  $\M_i \cap \F_c$ is in bijection with $\mathcal{J}_{i,c}.$ Fix $j \in \mathcal{J}_{i,c}$ and denote the associated components by $X^f_{c,j}$ and $\M^{\To}_{i,j}$ respectively.

\begin{enumerate}


	\item 
We have an isomorphism $X^f_{c,j} \simeq  \X_{l_j}.$ Further, the morphism $\delta$ restricts to a projection of a principal $\Jac(\X)[2]$-bundle  $\M^{\To}_{i,j} \rightarrow X^f_{c,j} $. Thus 
	\begin{align} \label{defMij}
	\M_{i,j}^{\To}=\{(E,\Phi_b) \in \M^f \cap \F_c : [\ord_x(b)]_p=j_x,  \ \forall x \in X^f\}.
	\end{align}
	\item
	For all $x\in X^f$ we have the equivariant decomposition 
	\begin{equation} \label{eq:directsumdecomp}	\Eb_{\mid \M_{i,j}^{\To}\times \{x\}} = \Eb_{c,1}(x) \otimes \sqrt{\hat{\alpha}_x}^{-1}\hat{\zeta}^{-\epsilon_{j,x}i_x}   \oplus \Eb_{c,0}(x) \otimes \sqrt{\hat{\alpha}_x}^{-1}\hat{\zeta}^{\epsilon_{j,x}i_x}. \end{equation}
	\end{enumerate} 
\end{lemma}

\begin{proof}
$1.$ The construction of the isomorphism 
\begin{equation} \label{X_c^f} \rho_c: X_{c}^f\rightarrow \bigsqcup_{j \in \mathcal{J}_c} \X_{l_j},
\end{equation}  is straightforward, and we have $X^f_{c,j}=\rho_c^{-1}(\X_{l_j})$ for all $j\in \mathcal{J}_c$. For  $\tilde{D}  \in \X_{l_j}$, the inverse is given by $\tilde{D}  \mapsto \pi^*(\tilde{D})+j.$ The proof that this is indeed an isomorphism is elementary.
		
		 We will now show  that $\delta: \F_c^f \rightarrow X_c^f$ is the projection of a  principal $\Jac(\X)[2]$-bundle. The decomposition given in equation \eqref{X_c^f}  will then entail that $\M_{i,j}^{\To}$, as defined in equation \eqref{defMij}, is in fact a principal $\Jac(\X)[2]$-bundle over $\X_{l_j}$. The proof of the equivariant decomposition given in equation \eqref{eq:directsumdecomp} will be given in the third step of the proof, and this will imply that $\M_{i,j}^{\To}$ is a subspace of $\M_i$. 
		
		First we notice that $\gcd(p,2)=1$ implies that $\pi^*(\Jac(X))[2]=\Jac(X)[2]^f$, and it is elementary to verify that $\pi^*$ defines an isomorphism. Now notice that the image of $\F_c^f$ under $\delta: \F_c \rightarrow X_c$ belong to $X_c^f$. Further, if $\delta(\mathcal{E}_1 ) =\delta(\mathcal{E}_2)$ for $\E_1,\E_2 \in \F_c^f$,  then there exists a unique $L \in \Jac(X)[2]$ with  $L. \mathcal{E}_1 = \mathcal{E}_2$. Then  \begin{align} f^*(L) . \mathcal{E}_1&= f^*(L).  f^*(\mathcal{E}_1)=f^*(L. \mathcal{E}_1)=f^*(\mathcal{E}_2)=\mathcal{E}_2=L. \mathcal{E}_1,
	\end{align} 
	and therefore the fact that $\F_c$ is a $\Jac(X)[2]$-torsor implies that $f^*(L)=L$. Thus $L \in \Jac(X)[2]^f$. Therefore the lemma will follow if we can establish surjectivity of the map $\delta: \F_c^f \rightarrow X_c^f$. Given $\delta_0 \in X_c^f$ we want to find $L_0 \in \Pic(X)^f$, that solves the equation
	\begin{equation} \label{eq:centralequation}
	L_0^2=\Lambda K_X (-\delta_0).
	\end{equation}
Recall that $K_X=\pi^*(K_{\X})(-(p-1)X^f)$. As $K_{\X}$ is of even degree, it has a  square-root, say $K_{\X}^{1/2}$. As $p$ is odd, it follows that we have an $\f$-equivariant square-root of $K_X$ given by 
\begin{equation}
K_{X}^{1/2}=\pi^*(K_{\X}^{1/2})(-2^{-1}(p-1)X^f).
\end{equation} Thus we have reduced the problem to finding an $\f$-equivariant square-root of $\Lambda (-\delta_0)$. Using the first part of the lemma, we see that there exists $j \in \mathcal{J}_{c}$ and  $D \in \X_{l_j}$, such that $\delta_0=\pi^*(D)+j$. As $p$ is an odd prime, we have $2 \in (\Z/p\Z)^{\times}$, and therefore we can for 
	each $x\in X^f$ find $a_x \in \Z$ such that
	$$
	2a_x= -j_x \mod p.
	$$
	Define $A=\sum_{x \in X^f} a_x \cdot x.$ We get from Proposition \ref{lem:uniquelift} that $\mathcal{O}_A$ admits a canonical $\f$-equivariant structure, such that $\f$-acts as the identity on the restriction of $\mathcal{O}_{-(j+2A)}$ to $X^f.$ In particular, we can define the quotient line bundle 
	\begin{equation} L=\mathcal{O}_{-(j+2A)}/\f \in \Pic(\X),
	\end{equation} and we have that $\mathcal{O}_{-(j+2A)}= \pi^*(L)$.  Recall that $\Lambda=\pi^*(\tilde{\Lambda})$. We observe that
	\begin{align}
	&\Lambda (-\delta_0)=\Lambda(-\pi^*(D)-j-2A)\mathcal{O}_{2A}=\pi^*(\tilde{\Lambda}L(-D))\mathcal{O}_{2A}
	\end{align} As $\mathcal{O}_A$ is an $\f$-equivariant square-root of $\mathcal{O}_{2A}$, we just need to find an $\f$-equivariant square-root of $\pi^*(\tilde{\Lambda}L(-D))$. As the degree of $\Lambda (-\delta_0)$ is even, and $\deg(\pi^*(\tilde{\Lambda}L(-D))=p \deg\tilde{\Lambda}L(-D)$ we see that the degree of $\tilde{\Lambda}L(-D)$ must be even. Hence it admits a square-root, say $(\tilde{\Lambda}L(-D))^{1/2}$, and therefore $\pi^*((\tilde{\Lambda}L(-D))^{1/2})$ is an $\f$-equivariant  square-root of  $\pi^*(\tilde{\Lambda}L(-D))$. This finishes the proof of the first part.

$2$. Recall the content of Remark \ref{rem:eigenvalues}. Let $x \in X^f$. As $f_{\Eb}$ commutes with the $\To$-action, and as $\To$ act with weight $t$ on $\Eb_{c,1}$ and with weight equal to unity on $\Eb_{c,0}$, it easily follows that we must have a $\To \times \langle f \rangle$-equivariant decomposition of the form \eqref{eq:directsumdecomp}, for some $\mu'_x \in \Z / p \Z$ in place of $-\epsilon_{j,x}i_x$. Thus we must prove that \begin{equation} \label{fixmuu} 2\mu'_x=-n_x(1+j_x) \mod p.
\end{equation} Let $\E=(L_1\oplus L_2, \Phi_b) \in \M_{i,j}^{\To}$ be given. Consider a frame $(s_1,s_2)$ near $x$, where $s_j$ is a frame of $L_j$ for $j=1,2$. Let $z$ be a coordinate centered at $x$. Write $b=\phi (s_2 \otimes s_1^* )  dz$, where $\phi\in \mathcal{O}_{X,x}$ is a regular function vanishing to order $d$ at $x$. Then $d= j_x \mod p$ by definition. Thus we may write $\phi=z^d \cdot h_1(z)$ with $h_1(0)\not=0.$ It follows that $f^*(\phi)=\zeta^{n_xj_x} z^df^*(h_1)$. As $d f^* (f^*(dz))=\frac{\partial f}{\partial z}dz=\zeta^{n_x}dz$, we obtain

\begin{equation} \label{eq:tilbagetraek}
(\Id \otimes d f^*) (f^*(\Phi_b))= \begin{pmatrix}
0 & 0\\ \zeta^{n_x(j_x+1)} z^df^*(h_1) & 0
\end{pmatrix} d z
\end{equation}
On the other hand, we may write  $f_{\E}$ with respect to the frames $(s_1,s_2)$ and $(f^*(s_1),f^*(s_2))$ as a diagonal matrix $f_{\E} =\diag(f_1,f_2)$, where $f_1,f_2 \in \mathcal{O}_{X,x}$ are regular functions defined near $x$ with $f_1(x)=\sqrt{\alpha}_x\zeta^{\mu'_x}$ and $f_2(x)=\sqrt{\alpha}_x\zeta^{-\mu'_x}$. Recall that $f^c_{\E}$ denotes conjugation by $f_{\E}$. Then we have that
\begin{align} \begin{split} \label{eq:konjugering}
& (f^c_{\E} \otimes \Id_{K})(\Phi_b) =\begin{pmatrix}
0 & 0
\\ f_2f_1^{-1} \phi & 0
\end{pmatrix} dz.
\end{split}
\end{align}
As $\E$ is a fixed point, equation \eqref{conjugation} holds for $\Phi_b$, and thus the right hand side of equation \eqref{eq:tilbagetraek} and the right hand side of equation \eqref{eq:konjugering} are equal, i.e. we have that
\begin{equation} \label{eq:okokok}
\zeta^{n_x(j_x+1)} z^df^*(h_1)  d z=  f_2f_1^{-1} z^d h_1  d z.
\end{equation}
Dividing out by $z^d$ on both sides of equation \eqref{eq:okokok}, evaluating at $z=0$ and using $h_1(0)\not=0$, we see that equation \eqref{fixmuu} holds.  \end{proof}

Fix $i \in \mathcal{I}.$ Towards the end of identifying the components of $\tM_i^{\To}$, we note that $\tilde{\mathcal{N}}_i$ is always one of the components, so it suffices to identify the components of $\tM_i^{\To}\setminus \tilde{\mathcal{N}}_i $. By the above lemma this is equivalent to identifying the image $\varXi_{i}(\M_{i,j}^{\To})$ for all $c \in C$ and $j \in \mathcal{J}_{i,c}.$ For all $c \in C$ and $j\in \mathcal{J}_{i,c}$ define the effective divisor $\D_{ij}$ on $\X$ by
\begin{equation} \label{def:Dij}
    \D_{ij}:=\sum_{x \in \mathcal{D}_i: \epsilon_{j,x}=-1} \pi(x).
\end{equation}
Below, we implicitly use the fact that the sheaf of holomorphic sections of $K_{\X}(\D_{i}-\D_{ij})$ is a subsheaf of the sheaf of holomorphic sections of $ K_{\X}(\D_{i})$. This fact implies that for any holomorphic line bundle $L \rightarrow \X,$ there is a natural inclusion 
\begin{equation} \label{eq:naturalinclusion} H^0(\X, LK_{\X}(\D_{i}-\D_{ij}))\hookrightarrow H^0(\X, LK_{\X}(\D_{i})).
\end{equation}
\begin{definition} \label{def:tF} 
       Set $\mathcal{J}'_i=\sqcup_{c \in C} \mathcal{J}_{i,c}$.  For  $j\in \mathcal{J}'_{i}$ denote by $\tM_{i,j}^{\To}$ the set of isomorphism classes of quasi parabolic Higgs bundles $(E,F,\Phi_b)$ on $(\X,\D_{i})$ of rank two of the following type. 
	\begin{itemize} \item There are $L_+,L_{-} \in \Pic(\X)$ with $E=L_+ \oplus L_{-} $ and $L_+ L_- \simeq \tilde{\Lambda}_{i}$, and for each $\tilde{x} \in \D_{i}$, we have that $F(\tilde{x})=L_{\epsilon_{j,x}}(\tilde{x})$.	\item The Higgs field $\Phi_b$ is the off-diagonal Higgs field associated with a section $b \in H^0(\X, L_+ L_{-}^{-1} K_{\X}(\D_{i}-\D_{ij}))$. The zero divisor of $b$ has degree $l_j$. 
		\end{itemize}
	\end{definition}

    \begin{remark}
       Consider a Higgs bundle of the form $(L_+\oplus L_{-},\Phi_b)$ as in Definition \ref{def:tF}. Let $x \in \mathcal{D}_i$ with $\epsilon_{j,x}=-1$. Then $L_{-}(\tilde{x})$ is the one-dimensional subspace specifying the flag. As $\Phi_b$ is the off-diagonal Higgs field associated with a section $b \in H^0(\X, L_+ L_{-}^{-1} K_{\X}(\D_{i}))$, the residue of the Higgs field $\Phi_b(\tilde{x})$ must necessarily map $L_{-}(\tilde{x})$ to $L_{+}(\tilde{x})$. By definition of parabolic Higgs bundles, the residue of the Higgs field is also nilpotent with respect to the flag. These two facts imply that the residue of the Higgs field must vanish at $\tilde{x}$. This is consistent with the fact that $b$ belong to the image of $ H^0(\X, L_+ L_{-}^{-1} K_{\X}(\D_{i}-\D_{ij}))$ under the natural inclusion \eqref{eq:naturalinclusion} and explains the role of the divisor $\D_{ij}$ defined in \eqref{def:Dij}.
    \end{remark}
\begin{corollary} \label{Cor:B}
	Fix $i \in \mathcal{I}$. Set $\tM^{\To}_{i,i}=\tilde{\mathcal{N}}_i$. Fix $j \in \mathcal{J}'_{i}$. Then $\varXi(\M^{\To}_{i,j})=\tM^{\To}_{i,j}$. Thus, the set $\mathcal{J}_i$ of components of $\tM_{i}^{\To}$ is in bijection with $\{i\}\sqcup\mathcal{J}'_i$. In particular, every $\E \in \tM_{i,j}^{\To}$ is stable with resepct to the weight $w_i$.  
\end{corollary}
\begin{proof} 
Recall the notation of Definition \ref{def:varXi}. Let $\E_0 \in \M_{i,j}^{\To}$ and write $\E_1=\E_0(S_i)$. The proof is divided into three parts. In the first, we analyze the effect of a single Hecke transform on $\E_i$. In the second part, we describe $\Psi_i(\E_1)$, and in the third and final part, we describe the Higgs bundle given by descent $\varXi_i(\E_0)=\Psi_i(\E_1)/ \f.$
\paragraph*{First part.} We begin by analyzing the effect of a single Hecke transformation. Write $\E_1=(L_1 \oplus L_2, \Phi_b)$, where $L_1,L_2$ are $\f$-equivariant line bundles and $\Phi_b$ is the off-diagonal Higgs field associated with a section $b \in H^0(X,L_2L_1^{-1}K_X)$. Consider the setup of the second step of the proof of Proposition  \ref{def:HeckeTransformations}. Let $x \in \mathcal{D}_i$. Assume for definiteness that $\epsilon_{j,x}=1$. We can choose the frame $(s_1,s_2)$ of $\E$ such that $s_1$ is a non-vanishing section of $L_1$ and $s_2$ is a non-vanishing section of $L_2$. Let $z$ be a local holomorphic coordinate on $X$ centered at $x$ and define $\phi\in \mathcal{O}_{X,x}$ to be the regular function such that 
\begin{equation} 
s_2^*(b(s_1))(z)=\phi(z)dz.
\end{equation} The Hecke transform is given as follows 
\begin{equation}
\E_1'=(L_1 \oplus L_2(-x), \Phi_{b'}),
\end{equation}
where $\Phi_{b}'$ is the Higgs field obtained by the meromorphic extension of $\iota^{-1}\circ\Phi_b \circ \iota$, where $\iota: \E_1' \rightarrow \E_1$ is the natural morphism described in Proposition \ref{def:HeckeTransformations}. Write $E=L_1 \oplus L_2$ and $E'=L_1 \oplus L_2(-x)$. Notice that in the frame $(s_1,s_2)$, the matrix of $f_{E}$ is diagonal in a neighbourhood of $x$. Therefore equation \eqref{F'conj} implies that the lift $f_{E'}$ of the Hecke transform $\E'$ is diagonal with respect to any frame  of the form $(s_1', s_2')$ where $s_1'$ is a non-vanishing section of $L_1$ and $s_2'$ is a non-vanishing section of $L_2(-x)$. Similarly, equation \eqref{eq:conjugatedHiggsField} implies that $\Phi'_b$ will also be an off-diagonal Higgs field associated with a section, i.e. with similar notation, it will be of the following form
\begin{equation}
    \Phi_b'= \begin{pmatrix}
        0 &0
        \\ z^{-1}\phi(z)& 0
    \end{pmatrix}\otimes dz, \quad (\epsilon_{j,x}=+1).
\end{equation}

The case that $\epsilon_{j,x}=-1$ is similar, except that in this case the modified bundle takes the form $\E=L_1(-x)\oplus L_2$ and the modified Higgs field is of the following form (with similar notation)
\begin{equation}
    \Phi_b'= \begin{pmatrix}
        0 &0
        \\ z\phi(z)& 0
    \end{pmatrix} \otimes dz, \quad (\epsilon_{j,x}=-1).
\end{equation}

We see that in each case $\Phi_{b}'$ can be described as the off-diagonal Higgs field associated with the section $b'=\iota^{-1}\circ b \circ \iota$, and we may write $\Phi_{b}'=\Phi_{b'}$. Moreover, in each case, it is clear that if we let $(F_{ab})_{1 \leq a,b \leq 2} \subset \mathcal{O}_{X,x}$ be the matrix of regular functions that describes the lift $f_{\E_1}$ with respect to the local frame $(s_1,s_2)$, then we see that $F_{12}$ and $F_{21}$ are constantly equal to zero. It follows that $(F_{ab})_{1\leq a,b\leq 2}$ is diagonal, and from remark \ref{rem:F_E'}, we deduce that the matrix of $F_{E'}$ with respect to the frame $(s_1',s_2')$ is also diagonal. It follows that the new frame $(s_1',s_2')$ may be used in a potential next Hecke tranformation, and this will be used below.
\paragraph*{Second part.} We now iterate the above analysis to describe $\Psi_i(\E_1)$. Define holomorphic line bundles defined as follows
\begin{align}
    &\Psi_{i,1}(L_1)
:=L_1\left(-\sum_{x \in X^f: \epsilon_{j,x}=-} b_i(x)\cdot x\right),
\\& \Psi_{i,2}(L_2):=L_2\left(-\sum_{x \in X^f: \epsilon_{j,x}=+} b_i(x)\cdot x\right).
\end{align} We see by induction that 
\begin{equation}
    \Psi_i(\E_1)=\prod_{x \in \D_i} \Psi_x^{b_i(x)}(\E_1)=(\Psi_{i,1}(L_1)\oplus \Psi_{i,2}(L_2), \Phi_{\Psi_i(b)}),
\end{equation}
where the first equality is simply the definition of $\Psi_i$, and where $\Phi_{\Psi_i(b)}$ is the off-diagonal Higgs field associated with the meromorphic section
\begin{equation} \Psi_i(b)\in H^0(X, \Psi_{i,1}(L_1)^{-1}\otimes \Psi_{i,2}(L_2)\otimes K_X),
\end{equation}
which is conjugate to $b$ on $X \setminus \D_i$, and which near each $x \in \D_i$ is of the following form 
\begin{equation} \label{eq:newb} (s''_2)^*(\Psi_i(b)(s_1''))(z)=z^{-\epsilon_{j,x}b_i(x)} \phi(z) dz,
\end{equation}
where $(s''_2,s''_1)$ is a pair of holomorphic frames of the holomorphic lines bundles defining the underlying holomorphic vector bundle of the Higgs bundle and $z$ is a local coordinate centered at $x$.
\paragraph*{Third part.} We now describe the effect of descent. By construction, the holomorphic line bundles $\Psi_{i,a}(L_a), a\in \{1,2\},$ and the Higgs field $\Phi_{\Psi(b)}$ are amenable to descent, and we write \begin{align}
    &L_+:=\Psi_{i,2}(L_2)/ \langle f \rangle,
    \\ & L_{-}:=\Psi_{i,1}(L_1)/ \langle f \rangle,
\end{align}
so that $\varXi_i(\E_0)=(L_+\oplus L_{-},\Phi_{\Psi_i(b)}/\f)$. We now argue that
\begin{equation} \label{eq:endlinebundle}
   \Psi_i(b)/ \f \in  H^0(\X, L_+ L_{-}^{-1} K_{\X}(\D_{i}-\D_{ij})).
\end{equation}
Recall the $\f$-invariant divisor $R_i=\mathcal{D}_i+(1-p)(X^f\setminus \mathcal{D}_i).$ By equation \eqref{phi'Rbeta} we have that
\begin{equation}
\label{eq:newlinebundle} \Psi_i(b) \in  H^0(X, \Psi_{i,1}(L_1)^{-1}\Psi_{i,2}(L_2)K_X(R_i))^{\f},
\end{equation} and by construction, and as argued in the proof of Theorem \ref{thm:M^f}, we have that \begin{equation} 
\Psi_i(b) / \f \in H^0(\X,L_+ L_{-}^{-1} K_{\X}(\D_{i})).
\end{equation}
To show \eqref{eq:endlinebundle} we will show that $\Psi_i(b)$ belong to the image of the natural morphism \eqref{eq:naturalinclusion}. We will now analyze the zero divisor of $\Psi_i(b)$. Recall that by Lemma \ref{lem:S^m(C)^f} there exists an effective divisor $\tilde{d}$ on $\X$ of degree $l_j$ such that the zero divisor of $b$ is of the form 
\begin{equation}
    (b)=\pi^*(\tilde{d})+j.
\end{equation}
It follows from \eqref{eq:newb} that the zero divisor of $\Psi_i(b)$ takes the form
\begin{align}
\begin{split} \label{eq:analysiszerodivisor} (\Psi_i(b))&=(b)+\sum_{x \in X^f} (1+\epsilon_{j,x}b_i(x))\cdot x-(p-1)(X^f/\mathcal{D}_i)
\\&=\pi^*(\tilde{d})+j+\sum_{x \in X^f} (1+\epsilon_{j,x}b_i(x))\cdot x-(p-1)(X^f/\mathcal{D}_i).\end{split}
\end{align}
Thus, to show that $\Psi_i(b)$ descends to a holomorphic section of \eqref{eq:endlinebundle} as desired, we must show that\begin{equation} \label{eq:Dij}
        \pi^*(\D_{ij})=j+\sum_{x \in \mathcal{D}_i}(1-\epsilon_{j,x}b_{i}(x))\cdot x-(p-1)(X^f \setminus \mathcal{D}_i).
    \end{equation}
For each $x \in X^f$ denote by $d_{ij}(x) \in \Z$ the coefficient of $x$ in the divisor defined by the right hand side of \eqref{eq:Dij}.  We must show that for each $x \in X^f$ we have that
   \begin{equation}      d_{ij}(x)= \begin{cases} p, &\text{if $\epsilon_{j,x}=-1$,}\\ 0, & \text{otherwise}.    \end{cases}   \end{equation}
    For $x \in X^f \setminus \mathcal{D}_i$, we have $i_x=0$, and therefore it follows directly from \eqref{ix2} that $j_x=p-1$. Hence $d_{ij}(x)=0$ in this case. Assume now that $x \in \mathcal{D}_i$ with $\epsilon_{j,x}=1$. Then $i_x\not=0 \mod p$. From \eqref{eq:epsilon} and the fact that $n_xm_x=-1$, we obtain $2i_xm_x=1+j_x \mod p$. As $j \in \{0,...,p-1\}$ we observe that $1+j_x \in \{1,...,p\}$. However, as $2i_xm_x =1+j_x \mod p$ and as $2i_xm_x \not=0 \mod p$, we have $1+j_x\in \{1,...,p-1\}$. Therefore $[2i_xm_x]_p=1+j_x$. As $b_i(x)=[2i_xm_x]_p$ by definition, it follows that
    \begin{equation}
        d_{ij}(x)=j_x+(1-\epsilon_{j,x}b_i(x))=j_x+(1-(1+j_x))=0, \quad (\epsilon_{j,x}=1).
    \end{equation}
    Assume now that $\epsilon_{j,x}=-1$. Then \eqref{eq:epsilon} and the fact that $n_xm_x=-1$ implies that $2i_xm_x = -(1+j_x) \mod p$. As $j_x \in \{0,...,p-1\}$, this implies that $b_i(x)=[2i_xm_x]_p=p-(1+j_x)$. Hence
    \begin{equation}
        d_{ij}(x)=j_x+(1-\epsilon_{j,x}b_i(x))=j_x+(1+(p-(1+j_x)))=p, \quad (\epsilon_{j,x}=-1).
    \end{equation}
    This finishes the argument that equation \eqref{eq:Dij} holds.


Finally, the fact that every $\E \in \tM_{i,j}^{\To}$ is stable with respect to  $w_i$ follows from Theorem \ref{thm:M^f}, the fact that $\M^{\To}_{i,j}$ is a sub-variety of $\M_i$ and the identity $\varXi(\M^{\To}_{i,j})=\tM^{\To}_{i,j}$.
\end{proof}

\begin{remark} \label{rem:intersectionpairings} From Corollary \ref{Cor:B}  it follows that every connected component is either a smooth moduli space of stable parabolic vector bundles on  $\X$ of rank two, or a principal-$\Jac(\X)[2]$ bundle over a symmetric power of $\X$. Consider first the case of moduli spaces of stable bundles. The intersection pairing of these moduli spaces have been thoroughly studied following Witten's work \cite{wittenTwoDimensionalGauge1992}. In the so-called coprime case, formulae for all pairings were proven in rank two in \cite{thaddeusConformalFieldTheory1992} and in all ranks in \cite{jeffreyIntersectionTheoryModuli1998}. Subsequently, a version of these formulae for a class of parabolic moduli spaces including $\tilde{\n}_i$ were proven in \cite{meinrenkenWittenFormulasIntersection2005}. Consider the case of the intersection pairings on the cohomology ring of symmetric powers of a Riemann surface. These rings are generated by MacDonald classes \cite{macdonaldSymmetricProductsAlgebraic1962} and all the relations in the cohomology ring was completely described in \cite{macdonaldSymmetricProductsAlgebraic1962}. This was used in \cite{thaddeusConformalFieldTheory1992} to compute certain characteristic numbers. 
	
\end{remark}

 \section{Localization} \label{sec:Localization}

We collect some generalities on localization before proving Theorem \ref{ThmMain}. Let $(M,\omega)$ be a Kähler manifold with a holomorphic $\To\times \mu_p$-action, for which the $\To$-action is meromorphic in the sense of  \cite{wuInstantonComplexHolomorphic2003}, i.e. assume Assumption $2.15$ in \cite[\S 2.2]{wuInstantonComplexHolomorphic2003} is satisfied. This is satisfied if $M$ is a semi-projective variety in the sense of \cite{bialynicki-birulaTheoremsActionsAlgebraic1973}. Let $h: M \rightarrow M$ be a generator of the $\mu_p$-action. Let $L\rightarrow M$ be a holomorphic $\To\times \mu_p$-equivariant line bundle of curvature $-i\omega$. Then $L \rightarrow M$ is a quantum line bundle for $(M,\omega)$ in the sense of geometric quantization \cite{kostantQuantizationUnitaryRepresentations1970,souriauStructureSystemesDynamiques1970}. For all $q \in \Z_{ +}$ and $m \in \Z$ denote by $H^q_m(M,L)$ the $t^m$-weight sub-space of $H^q(M,L)$. Finite dimensionality of these weight sub-spaces follows from the results of \cite{wuInstantonComplexHolomorphic2003}. 
Define
	\begin{equation}
	\chi_{\To}(M,L,h)(t)= \sum_{m \in \Z} \sum_{q \geq 0} (-1)^q \tr( h: H_m^q(M,L)\rightarrow H_m^q(M,L)) \cdot t^m.
	\end{equation}
Proposition \ref{pro:generalloc} below gives a general localization formula for $\chi_{\To}(M,L,h)$ which we will apply to the automorphism equivariant Hitchin index as defined in \eqref{eq:index}. To state Proposition \ref{pro:generalloc}, we first introduce some notation.
    
   The meromorphicity condition on the $\To$-action on $M$ implies that the connected components of $M^{\To}$ are compact smooth complex manifolds. For our purpose, we can and will assume each of these components is in fact a smooth projective variety. Write $A:=\pi_0(M^{\To \times \mu_p})$ and denote by $\{M_a\}_{a \in A} $ the components of $M^{\To \times \mu_p}$. Fix  $a\in A$. Let $T_a M$ denote the restriction of $TM$ to $M_a$. For every $\To\times \mu_p$-equivariant vector bundle $W\rightarrow M_a$ denote by $W^{\pm}$ the $\pm$-$\To$-weight subbundle of $W$, and define $\nu_{a}= \rank(T_a M^{-}).$ Recall the notation $s$ for the equivariant sum of all symmetric powers. Recall that for any algebraic group $H$, we denote by $K_H^0$ the functor that assigns the K-theory ring of $H$-equivariant algebraic vector bundles to a smooth projective $H$-variety. Define the equivariant class
    \begin{equation} \label{defomega}
  \omega(W)=s((W^{+})^*) \cdot s(W^{-}) \cdot \det(W^{-}) \in  K^0_{\To \times \mu_p}(M_a)[[t^{-1}]].
  	\end{equation}  
The equivariant operation $\Omega$ introduced in \eqref{def:Omega} in Section \ref{sec:introlocalization} play a key role in Proposition \ref{pro:generalloc} and we now justify why it is well-defined. Denote by $\lambda=\sum_{l\geq 0} (-1)^l \lambda^l$ the equivariant operation of taking the alternating sum of exterior powers. Let $W^{\To} \subset W$ (resp. $W^{\To\times \mu_p}\subset W)$ denote the sub-bundle fixed by $\To$ (resp. $\To\times \mu_p$). Recall from Section \ref{sec:introlocalization} the multplicatively closed subset $S \subset \Z[\mu_p]$ generated by the elements $1-\hat{\zeta^l}$ where $l$ runs through the set $\{1,...,p-1\}.$ For any smooth projective $H$-variety $Z'$ we observe that $K_{\mu_p}(Z')$ is an algebra over the ring $\Z[\mu_p]$, and we can therefore regard $S$ as a multiplicatively closed subset of $K^0_{\mu_p}(Z')$. To justify why $\Omega(W)$ is well-defined, we must justify the invertibility of the image of the class $\lambda((W^\To / W^{\To \times \mu_p})^*)$ under the natural homomorphism of rings $K^0_{\mu_p}(M_a)\rightarrow S^{-1}  \cdot K^0_{\mu_p}(M_a)\otimes \Q$. This follows from Proposition \ref{pro:Inversion} below.  We now make the simple observation that we may rewrite $\Omega(W)$ in terms of $\omega(W)$ as follows
\begin{equation} \label{eq:Omegaintermsofomega}
    \Omega(W)= \frac{\omega(W)}{\lambda((W^\To / W^{\To \times \mu_p})^*)} \in S^{-1} \cdot K^0_{\To \times \mu_p}(M_a)\otimes_{\Z} \Q[[t^{-1}]].
\end{equation} 
Further, we observe that $\Omega$ is a homomorphism that maps sums to products. 
\begin{proposition} \label{pro:generalloc}  We have that
		\begin{align} \begin{split}  \label{eq:tcomponentdecompositionG}
	&\chi_{\To}(M,L,h) =
	\sum_{a \in A} (-1)^{\nu_{a}}\int_{M_a} 
	\ch  \left(\Omega(T_aM)\cdot L_{\mid M_{a}} \right) \Td M_{a}. \end{split}
	\end{align}
	\end{proposition}  Proposition \ref{pro:generalloc} follows essentially from a combination of a $\To$-localization result of Wu \cite{wuInstantonComplexHolomorphic2003} with the Atiyah-Bott fixed point formula \cite{atiyahIndexEllipticOperators1968}, also known as the Lefshetz fixed-point formula. In \cite{wuInstantonComplexHolomorphic2003}, Wu considers the case where $\mu_p$ act trivially on $M$. In this case $\{M_a\}_{a \in A}$ is equal to the set of components of $M^{\To}$ and in \cite[\S 4.2]{wuInstantonComplexHolomorphic2003} the following is proven.
\begin{theorem}[\cite{wuInstantonComplexHolomorphic2003}] 
		\label{thmWu}  In $\To$-equivariant $K$-theory we have 
\begin{align} \begin{split} \label{eq:IdentificationKTheory2}
&\sum_{\substack{m \in \Z, \\ q\geq 0}}  (-1)^q H^q_m(M,L) \cdot  t^m= \sum_{\substack{m \in \Z, \\ q \geq 0
		}} \sum_{a \in A} (-1)^{\nu_{a}+q} H_m^q(M_a, \omega(T_a;)\cdot L_{\mid M_a}) \cdot t^m.
\end{split} 
\end{align}
\end{theorem}

We now recall the Atiyah-Bott Lefshetz formula \cite{atiyahIndexEllipticOperators1968}. Let $Z$ be a smooth projective complex variety with an algebraic $\mu_p$-action generated by an isomorphism $h:Z \rightarrow Z$. Let $W \rightarrow Z$ be a $\mu_p$-equivariant algebraic vector bundle. Consider the equivariant index
$$\chi(Z,W,h)=\sum_{q\geq 0} (-1)^q \tr( h : H^q(Z,W)\rightarrow H^q(Z,W)).$$ 
Let $Z^{\mu_p}$ denote the fixed locus. Write $B:=\pi_0(Z^{\mu_p})$ and denote by $\{Z_b\}_{b \in B}$ the components of $Z^{\mu_p}$. For each component $Z_b$, denote by $N^*(Z_b)$ the conormal of $Z_b\subset Z$. The components are of course smooth and projective. Fix $b \in B$. We have natural isomorphisms \begin{equation} K^0_{\mu_p}(Z_b)\cong\Z[\mu_p]\otimes_{\Z} K^0(Z_b) \cong \bigoplus_{l=0}^{p-1} \hat{\zeta}^l \cdot K^0(Z_b),
\end{equation} the leftmost of which is an isomorphism of $\Z[\mu_p]$-algebras and the rightmost of which is an isomorphism of abelian groups. For each $l \in \{0,...,p-1\}$, we denote by $\pi_l:K^0_{\mu_p}(Z_b) \rightarrow K^0(Z_b)$ the projection onto the coefficient of $\hat{\zeta}^l$. Similarly, for an equivariant algebraic vector bundle $V \rightarrow Z_b$, we also use the notation $\pi_l(V)$ to denote the algebraic vector sub-bundle of $V \rightarrow Z_b$ on which $\mu_p$ act with weight $\hat{\zeta}^l$. Recall the multiplicatively closed subset $S \subset \Z[\mu_p]$ generated by the elements $1-\hat{\zeta}^l$, where $l$ runs through the set $\{1,...,p-1\}.$
\begin{proposition} \label{pro:Inversion}
Let $V\rightarrow Z_b$ be a $\mu_p$-equivariant algebraic vector bundle. Then $\lambda(V / V^{\mu_p}))$ is invertible in $S^{-1}\cdot K^0_{\mu_p}(Z)\otimes \Q$. Moreover, the equivariant rank $\rank(\lambda(V/ V^{\mu_p}))$ is an element of $S$, and if we set $r_l= \rank(\pi_l(V))$, then the following two formuale holds
\begin{align} \label{eq:inversionformula} \begin{split} & \rank(\lambda(V/ V^{\mu_p}))=\prod_{l=1}^{p-1} (1-\hat{\zeta}^l)^{r_l},
\\&    \lambda(V / V^{\mu_p})^{-1}=\rank(\lambda(V/ V^{\mu_p}))^{-1}\sum_{a=0}^{\dim(Z_b)}\left(\frac{(\rank(\lambda(V/ V^{\mu_p}))-\lambda(V / V^{\mu_p})}{\rank(\lambda(V/ V^{\mu_p}))}\right)^a.
    \end{split}
\end{align}
\end{proposition}
The first part of this proposition is often stated in the case where $V=T^*_{\mid Z_b}Z$, so that $V/ V^{\mu_p}$ is the conormal of $Z_b \hookrightarrow Z$. For the sake of completeness, we present a proof.
\begin{proof} Let $A(Z_b)$ be the Chow ring of $Z_b$. By \cite[Example 15.2.16]{fultonIntersectionTheory1984}) the Chern character induces an isomorphism $ \ch: K^0(Z_b)\otimes_{\Z}\Q \overset{\sim}{\rightarrow} A(Z_b)\otimes_{\Z}\Q.$    Clearly, in our setting, this implies that the equivariant Chern character induces an isomorphism 
    \begin{equation} \label{eq:isoKA}
     \ch: S^{-1} \cdot K_{\mu_p}^0(Z_b)\otimes_{\Z}\Q \overset{\sim}{\rightarrow} S^{-1}\cdot \Z[\mu_p]\otimes_{\Z} A(Z_b)\otimes_{\Z}\Q.
    \end{equation}
    We will refer to the target ring of this isomorphism as the equivariant Chow ring. Define algebraic vector bundles $V_l$ where $l$ runs through the set $\{1,...,p-1\}$ by $\pi_l(V)=V_l$. As $\lambda$ is a homomorphism we have that 
    \begin{equation}
    \lambda(V / V^{ \mu_p})= \prod_{l=1}^{p-1} \lambda(\hat{\zeta}^l\cdot V_l).
    \end{equation}
    It is enough to argue that for each $l \in \{1,...,p-1\}$ the factor $\lambda(\hat{\zeta}^l\cdot V_l)$ is invertible. Therefore, we will rename $V_l=V$ and set $r=\rank(V_l).$ By the splitting principle \cite[Section 3.2]{fultonIntersectionTheory1984}, there exists a smooth projective variety $Z'$ and a flat morphism $q:Z'\rightarrow Z_b$ such that $q^*:A^*(Z_b) \rightarrow A^*(Z')$ is injective and $q^*(V)\rightarrow Z'$ splits into a direct sum of line bundles, say $q^*(V)=\oplus_{b=1}^r L_b$. In the Chow ring we have that
    \begin{equation} \label{eq:equirank2}
\ch_0(\lambda(V))=\ch_0\left(\prod_{b=1}^r\lambda(L_b)\right)=(1-\hat{\zeta}^l)^r \in S.
    \end{equation}
Further, in the Chow ring we have that
\begin{equation} \label{eq:suminvertibleandnilpotent}
    \ch(\lambda(V))=\ch_0(\lambda(V))-(\ch_0(\lambda(V))-\ch(\lambda(V))).
\end{equation}
Since $\ch_0(\lambda(V)) \in S$ is invertible in the equivariant Chow ring, and since the element $(\ch_0(\lambda(V))-\ch(\lambda(V)))$ nilpotent and $A^*(Z_b)$ is a complete graded ring with highest grade equal to $\dim(Z_b)$, it is clear from \eqref{eq:suminvertibleandnilpotent} that $\ch(\lambda(V))$ is invertible with inverse given by
\begin{align} \label{eq:geometricseries}\begin{split}
&\ch(\lambda(V))^{-1}=\ch_0(\lambda(V))^{-1}\left(1-\frac{\ch_0(\lambda(V))-\ch(\lambda(V))}{\ch_0(\lambda(V))}\right)^{-1}
\\&=\ch_0(\lambda(V))^{-1}\sum_{a=0}^{\dim(Z_b)}\left(\frac{\ch_0(\lambda(V))-\ch(\lambda(V))}{\ch_0(\lambda(V))}\right)^{a}.
   \end{split}
\end{align} Since $\ch(\rank(W))=\ch_0(W)$ for any equivariant vector bundle $W$ and since \eqref{eq:isoKA} is an isomorphism of rings we see that the identities in \eqref{eq:inversionformula} follows from \eqref{eq:equirank2} and \eqref{eq:geometricseries}. Our reason for resorting to the isomorphism \eqref{eq:isoKA} is simply that the grading on $A^*(Z_b)$ makes it transparent that $(\ch_0(\lambda(V))-\ch(\lambda(V))$ is nilpotent.\end{proof}

We now state the Atiyah-Bott Lefshetz formula. Note that by Proposition \ref{pro:Inversion} the class $\lambda(N^*(Z_b))$ is invertible in $S^{-1} \cdot K_{\mu_p}^0(Z_b)$ for each $b \in B$.
\begin{theorem}[\cite{atiyahIndexEllipticOperators1968}]  We have that
	\begin{align}  \begin{split} \label{Lefshetz}
&\chi(Z,W,h)= \sum_{b \in B}\int \limits_{Z_b} \ch\left(\frac{W_{\mid Z_b}}{\lambda(N^*(Z_b))} \right) \Td Z_b.
\end{split} 
\end{align}
\end{theorem}

\begin{proof}[Proof of Proposition \ref{pro:generalloc}]
	The proof is a simple computation which combines \eqref{eq:IdentificationKTheory2}  with \eqref{Lefshetz}. The key point is that the identity \eqref{eq:IdentificationKTheory2} can be generalized to an identity in the $K$-theory ring of $\To \times \mu_p$-representations.
\end{proof}
We now recall Nielsens' localization result \cite{nielsenDiagonalizablyLinearizedCoherent1974}. As above, we can regard $S$ as a multiplicatively closed subset of $K^0_{\mu_p}(Z)$ and of $K^0_{\mu_p}(Z^{\mu_p})$.

\begin{theorem}[\cite{nielsenDiagonalizablyLinearizedCoherent1974}]  \label{thmNielsen} Pull-back along the inclusion $i: Z^{\mu_p} \rightarrow Z$ induces an isomorphism $S^{-1} \cdot K_{\mu_p}^0(Z)\overset{\sim}{\rightarrow}S^{-1} \cdot  K_{\mu_p}^0(Z^{\mu_p})$. The inverse is given by 
	\begin{equation} \label{Inverse}
	W \mapsto i_!(W \cdot \lambda^{-1}(N^*(Z^{\mu_p}))).
	\end{equation}
\end{theorem} 
We now analyze $K_{\mu_p}^0(Z^{\mu_p})$ in more detail. Recall that we have canonical isomorphisms
\begin{align}
& \label{iso:ring}K_{\mu_p}^0(Z^{\mu_p}) \cong K^0(Z^{\mu_p})\bigotimes_{\Z} \Z[\mu_p]  \cong \bigoplus_{j=0}^{p-1} \K^0(Z^{\mu_p}) \cdot \hat{\zeta}^j.
\end{align}
 Recall that for $j\in \{0,...,p-1\}$ we write $\pi_j$ for the projection onto the $j$'th summand in \eqref{iso:ring}. We introduce the special notation $W_0=\pi_0(W).$  We denote the universal localization homomorphism by $L: K^0_{\mu_p}(Z^{\mu_p}) \rightarrow S^{-1} \cdot K^0_{\mu_p}(Z^{\mu_p}).$ Recall that the kernel of $L$ is generated by all elements $r$ such that $s\cdot r=0$ for some $s \in S$. Let $Y$ be a formal variable. Recall that $\Z[\mu_p]\cong \Z[Y]/(Y^P-1)$ and that $Y^P-1=(Y-1)\Phi_p(Y)$ is the factorization of $Y^p-1$ into irreducible factors. As $(1-\hat{\zeta})\in S$,  it is therefore clear that we have $\ker(L)= K^0(Z^{\mu_p})\otimes_{\Z} (\Phi_p(\hat{\zeta})).$ Moreover, the following sequence of abelian groups is split exact\begin{equation}	0 \rightarrow \ker(L) \rightarrow K_{\mu_p}^0(Z^{\mu_p}) \overset{L}{\rightarrow}\im(L)\rightarrow 0,	\end{equation}
and a complement to $\ker(L)$ is given by the module  generated by all classes of the form
$W-W_0\cdot \Phi_p(\hat{\zeta})$, where $W$ ranges over all equivariant bundles. The following lemma is an elementary consequence of these facts.
\begin{lemma} \label{nytLemma} 

	If $V, W\in K_{\mu_p}^0(Z^{\mu_p}) $ are classes such that $V=W \mod \ker(L)$, then the following identity holds in $K_{\mu_p}^0(Z^{\mu_p})$
	\begin{equation}
V-V_0\cdot \Phi_p(\hat{\zeta})=W- W_0 \cdot \Phi_p(\hat{\zeta}).
	\end{equation}
\end{lemma}

\subsection{Application of Localization} Recall now the setup and the notation introduced in Section \ref{sec:introlocalization}. For each $l \in \{0,...,p-1\}$, let $\pi_l$ be the projection onto the $l$'th summand of $R_{i,j}$ in \eqref{decomp}.

\begin{definition}\label{def:zetagamma} For all $i \in \mathcal{I}$, all $j\in \tilde{\mathcal{J}_i}$ and all $l \in \{1,...,p-1\}$ define the integer $\tilde{\nu}_{i,j}$ and the rationals $\tilde{a}_{i,j,l}$ and $\tilde{b}_{i,j,l}$ by the formulae
	\begin{align}
	&\tilde{\nu}_{i,j}=\rank(T_{i,j}\tM_i^{-})+ \ch_0( \tilde{V}_{i,j}^{-}-\pi_0(\tilde{V}_{i,j}^{-})),
\\ 	&\tilde{a}_{i,j,l}=\rank(T\tM^{\To}_{i,j})+\frac{\ch_0\left( \pi_l(\tilde{V}^{\To}_{i,j}) -\pi_0(\tilde{V}^{\To}_{i,j})\right)}{p},
	\\& \tilde{b}_{i,j,l}=\rank(T_{i,j} \tM_i^{-}) + \frac{\ch_0\left( \pi_l(\tilde{V}^{-}_{i,j}) -\pi_0(\tilde{V}^{-}_{i,j})\right)}{p}.
	\end{align}
	Set $\tilde{\rho}_{i,j}= \sum_{l=1}^{p-1} l\cdot \tilde{b}_{i,j,l}$ and define
	$ \tilde{\zeta}_{i,j}= \hat{\zeta}^{\tilde{\rho}_{i,j}}\prod_{l=1}^{p-1}(1-\hat{\zeta}^l)^{-\tilde{a}_{i,j,l}}.$
\end{definition}
It will be seen in the proof of Theorem \ref{ThmMain} that the $\tilde{a}_{i,j,l}$ and the $\tilde{b}_{i,j,l}$ are in fact positive integers. Before given the proof of Theorem \ref{ThmMain} we state the following Lemma.
\begin{lemma} \label{lemmapthroot} Let $i\in \mathcal{I}$ and $j \in \mathcal{J}_i$. Set $d=\dim_{\C}(\M_{i,j}^{\To})$ and define the ring $R=\oplus_{l=0}^dH^{2l} (\M_{i,j}^{\To},\Q)\otimes_{\Z} (S^{-1}\cdot \Z[\mu_p])$. Assume $a,b \in R[[t^{-1}]]$ satisfy $a^p=b^p$, and that the degree zero components (with respect to the cohomological-grading) of the leading terms (with respect to the $t^{-1}$-grading)  of $a$ and $b$ are equal and non-zero. Then $a=b$.	\end{lemma} 

\begin{proof}
	Write $a=\sum_{i=0}^{\infty} a_i t^{-i}$ and $b=\sum_{i=0}^{\infty} b_i t^{-i}$ with $a_i,b_i \in R$ for all non-negative integers $i$. Without loss of generality, we may assume $a_0\not=0$  and $b_0\not=0$. We have
	\begin{align} \begin{split} \label{algebra0}
	&a^p=\sum_{i=0}^{\infty} \sum_{j_1+\cdots +j_p=i} a_{j_1}\cdots a_{j_p} t^{-i}=b^p=\sum_{i=0}^{\infty} \sum_{j_1+\cdots+ j_p=i} b_{j_1}\cdots b_{j_p} t^{-i}.
	\end{split}
	\end{align}
	In particular, we have that $a_0^p=b_0^p$. For each positive integer $l$, let $R^l$ denote the degree $2l$ part of $R$ with respect to the cohomological grading.  Write $a_0=\sum_{l=0}^d a_{0,l}$ and $b_0=\sum_{l=0}^d b_{0,l}$ with $a_{0,l}, b_{0,l} \in R^l$ for $l=0,...,d$. Then we have $a_{0,0}=b_{0,0}$ by assumption, and we see that 
		\begin{align} \begin{split}  \label{algebra}
	&a_0^p=\sum_{l=0}^{d} \sum_{l_1+\cdots +l_p=l} a_{0,l_1}\cdots a_{0,l_p} =b_0^p=\sum_{i=0}^{d} \sum_{l_1+\cdots+ l_p=l} b_{0,l_1}\cdots b_{0,l_p}.
	\end{split} 
	\end{align} 
	 For each $l=1,..,d$ we get by taking the degree $l$ part of \eqref{algebra} that
		\begin{align} \begin{split}  \label{algebra2}
	&\sum_{l_1+\cdots +l_p=l} a_{0,l_1}\cdots a_{0,l_p}= pa_{0,0}^{p-1}a_{0,l}+\sum_{\substack{l_1+\cdots +l_p=l,
			\\ l_t <l, t=1,...,p}} a_{0,l_1}\cdots a_{0,l_p}
	\\&= \sum_{l_1+\cdots +l_p=l} b_{0,l_1}\cdots b_{0,l_p}= pb_{0,0}^{p-1}b_{0,l}+\sum_{\substack{l_1+\cdots +l_p=l,
			\\ l_t <l, t=1,...,p}} b_{0,l_1}\cdots b_{0,l_p}.
	\end{split} 
	\end{align}
	Now notice that any element of $R$ whose projection onto $R^0$ is non-zero is not a zero divisor in $R$. In particular $a_{0,0}^p=b_{0,0}^p$ is not a zero divisor, and therefore we get by induction from \eqref{algebra2} that $a_{0,l}=b_{0,l}$ for $l=0,...,d$. Thus $a_0=b_0$. From \eqref{algebra0} we deduce that for all non-negative integers $i$ we have that
	\begin{align} \begin{split}  \label{algebra3}
&\sum_{j_1+\cdots +j_p=i} a_{j_1}\cdots a_{j_p}= pa_0^{p-1}a_i+\sum_{\substack{j_1+\cdots +j_p=i,
\\ j_t <i, t=1,...,p}} a_{j_1}\cdots a_{j_p}
\\&= \sum_{j_1+\cdots +j_p=i} b_{j_1}\cdots b_{j_p}= pb_0^{p-1}b_i+\sum_{\substack{j_1+\cdots +j_p=i,
		\\ j_t <i, t=1,...,p}} b_{j_1}\cdots b_{j_p}.
	\end{split} 
	\end{align}
	By induction, and the fact that $a_0^p$ and $b_0^p$ are not zero divisors, we obtain from \eqref{algebra3} that $a_i=b_i$ for all non-negative integers $i$. Thus $a=b$. \end{proof}  

\begin{proof}[Proof of Theorem \ref{ThmMain}] 

	We must prove equation \eqref{eq:Main}. The bundle $\LD \rightarrow \M$ satisfies the conditions of Proposition \ref{pro:generalloc} with the $\To \times \langle f \rangle$-equivariant structure constructed in Section \ref{sec:Lift}. This follows from the fact that $\M$ is a semi-projective variety \cite{HV15}. Thus we can apply Proposition \ref{pro:generalloc}. Let  $i\in \mathcal{I}$ and $j \in \mathcal{J}_i$. To ease notation, we will write $\M_{i,j}^{\To}=\M_{i,j}$ and we will write $T_{i,j} \M$ (resp. $T_{i,j} \M_i$) for the restriction of $T\M$ (resp. $T\M_i$) to $\M_{i,j}$. By applying Proposition \ref{pro:generalloc} we obtain
    \begin{equation}
        \chi_{\To}(\M,\LD^{k},f)=\sum_{i \in \mathcal{I}}\sum_{j \in \mathcal{J}_i} (-1)^{\nu_{i,j}} \int_{\M_{i,j}} \Omega(T_{i,j} \M)\cdot  \ch(\LD^k_{\M_{i,j}}) \cdot \Td (\M_{i,j}).
    \end{equation} Therefore, proving equation \eqref{eq:Main} will be sufficient to prove that for all $i \in \mathcal{I}$ and $j \in \mathcal{J}_i$ we have that
	\begin{align} \label{eq:Closed}	 (-1)^{\nu_{i,j}}\Omega(T_{i,j} \M) &=(-1)^{\tilde{\nu}_{i,j}}\varXi_i^*(\tilde{\chi}_{i,j}),
  \\ \label{eq:inducedlinebundle} \ch(\LD_{\mid \M_{i,j}})&=\hat{\zeta}^{m_i}\varXi_i^*(\exp(\tilde{\omega}_{i,j})).
	\end{align} 
    This naturally divides the rest of the proof into two parts. In first, we will prove \eqref{eq:Closed}, in the second, we will prove \eqref{eq:inducedlinebundle}.

    \paragraph*{First part.} We will begin by proving the equation obtained by raising both sides of equation \eqref{eq:Closed} to the $p'th$ power. Write $\mathbb{U}=\End_0(\Eb) $, and define the equivariant class
	\begin{equation}
	V_{i,j}=\sum_{x \in X^f}\mathbb{U}_{\mid \M_{i,j}\times \{x\}} \cdot(1-t\cdot\hat{\zeta}^{-n_x})\sum_{l=1}^{p-1} l \cdot \hat{\zeta}^{-l n_x} .
	\end{equation}
	Then $\varXi_i^*(\tilde{V}_{i,j})=V_{i,j}$. Let $\pi: \M^{\To}\times X \rightarrow \M^{\To}$ be the projection onto the first factor. On $\M_{i,j} \subset \M^{\To}$, we have the following equation in $\To$-equivariant $K$-theory 
\begin{equation} \label{eq:derived}
T_{i,j} \M=\pi_!(\mathbb{U} \cdot \pi^*(K_X)\cdot t-\mathbb{U}).
\end{equation} This can be seen for instance by considering a standard long exact sequence for hypercohomology groups of two-term chain complexes (see \cite[\S 2.2.3]{hauselGeometryModuliSpace2001}). Consider the natural projection onto the first factor $\pi': \M_{i,j} \times X^f \rightarrow \M_{i,j}$. Note that we have  $(\M_{i,j}\times X)^f=\M_{i,j}\times X^f$. Consider the inclusion $\iota: \M_{i,j} \times X^f \rightarrow \M_{i,j}\times X$. Clearly $\pi\circ\iota=\pi'.$ By applying Nielsen's localization theorem to the right hand side of quation  \eqref{eq:derived}, we obtain the following identity in $S^{-1}\cdot K^0_{\To \times \mu_p}(\M_{i,j})$
	\begin{align}
T_{i,j} \M&=\pi_!\left(\iota_!\left(\frac{ \iota^*(\mathbb{U}(\pi^*(K_X)\cdot t-1))}{\lambda(N^*(\M_{i,j}\times X^f))} \right)\right)= \pi'_!\left(\frac{ \iota^*(\mathbb{U}(\pi^*(K_X)\cdot t-1))}{\lambda(N^*(\M_{i,j}\times X^f))} \right)
\\&= \sum_{x \in X^f }\frac{\mathbb{U}_{\mid \M_{i,j}\times \{x\}} \cdot(1-t\cdot\hat{\zeta}^{-n_x})}{\hat{\zeta}^{-n_x}-1}.
	\end{align}
		However, the  $\To\times \mu_p$-equivariant $K$-theory operation $s$ of taking the sum of all symmetric powers is not well-defined in $S^{-1}\cdot K^0_{\To \times \mu_p}(\M_{i,j})$, and we need to extract an identity in the ring  $K_{\To\times \mu_p}^0(\M_{i,j})$. For all $u=1,...,p-1$, we have the following identity in $S^{-1}\cdot \Z[\mu_p]$ 
	\begin{equation} \label{eq:nyttig}
		p \cdot (\hat{\zeta}^{u}-1)^{-1}=\sum_{l=1}^{p-1} l \hat{\zeta}^{ul},
	\end{equation} 
	which easily follows from the fact that $\Phi_p(\hat{\zeta}^{u})=0$ in $S^{-1}\cdot \Z[\mu_p]$ for $u=1,...,p-1$. Multiplying by $p$ and using equation \eqref{eq:nyttig} we obtain the following identity in $K_{\To\times \mu_p}^0(\M_{i,j}) / \ker(L)$
	\begin{equation} \label{over}
	p\cdot T_{i,j} \M= V_{i,j}  \mod \ker(L).
	\end{equation}
Notice that $\pi_0(T_{i,j}\M) = T_{i,j} \M_i$. Thus equation \eqref{over} together with Lemma \ref{nytLemma} implies that the following identity holds in the ring  $K_{\To\times \mu_p}^0(\M^{\To}_{i,j}) $
\begin{equation} \label{eq:Vigtig}
p\cdot T_{i,j} \M=p \cdot T_{i,j} \M_i \cdot \Phi_p(\hat{\zeta}) +V_{i,j}- \pi_0( V_{i,j}) \cdot \Phi_p(\hat{\zeta}).
\end{equation}
We can now apply $\Omega$ to equation \eqref{eq:Vigtig}. Using that $\Omega$ is a homomorphism that transforms sums into products, we obtain the desired consequence of equation \eqref{eq:Closed}
\begin{equation} \label{a}
\Omega(T_{i,j} \M)^p=\Omega( T_{i,j} \M_i \cdot \Phi_p(\hat{\zeta}))^p\frac{\Omega(V_{i,j})}{\Omega(\pi_0(V_{i,j}) \cdot \Phi_p(\hat{\zeta}) )}.
\end{equation}

Since $p$ is odd, it is trivial to see that equation \eqref{eq:Vigtig} implies that $(-1)^{\nu_{i,j}}=(-1)^{\tilde{\nu}_{i,j}}.$ Thus it only remains to extract the correct $p$'th root. Let $N^*_{i,j}$ be the conormal of $\M_{i,j}\subset \M^{\To}$. For each $l=1,...,p-1$ define $a_{i,j,l}=\rank(\pi_l(N^*_{i,j}))$, and define $b_{i,j,l}=\rank(\pi_l(T_{i,j} \M^{-}))$. Further, define $\rho_{i,j}=\sum_{l=1}^{p-1} b_{i,j,l} \cdot l.$ Consider the projection, which takes $\ch_0$ of the leading term of a $1/t$-series  
\begin{equation}
P_0: H^*(\M_{i,j},\Q)\otimes \Q[\mu_p][[t^{-1}]] \rightarrow H^0(\M_{i,j},\Q) \otimes (S^{-1}\cdot \Z[\mu_p]).
\end{equation} Upon inspecting the definition of $\omega$ and using \eqref{eq:Omegaintermsofomega} one finds that
\begin{align} \begin{split}\label{eq:leadingterm}
P_0(\ch(\Omega(T_{i,j} \M) ))&=P_0\left(\ch(\omega(T_{i,j} \M) \cdot \lambda\left(N^*_{i,j} )\right)^{-1}\right)
\\&=\hat{\zeta}^{\rho_{i,j}} \prod_{l=1}^{p-1} (1-\hat{\zeta}^{l})^{-a_{i,j,l}}.
\end{split}
\end{align}
From equation \eqref{eq:Vigtig} we deduce that 
\begin{align}
&p \left( T_{i,j} \M^{\To} - T \M_{i,j}\right) =p \sum_{l=1}^{p-1}   T \M_{i,j} \cdot \hat{\zeta}^l+ V^{\To}_{i,j} -\pi_0(V_{ i,j }^{\To} )\cdot \Phi_p(\hat{\zeta}) ,
\\& p T_{i,j}\M^{-}=p  (T_{i,j} \M_i)^{-} \cdot  \Phi_p(\hat{\zeta}) +V_{i,j}^{-}-\pi_0(V_{i,j})^{-}\cdot \Phi_p(\hat{\zeta}),
\end{align}
and therefore, we obtain
\begin{align}
&a_{i,j,l}= \dim(\M^{\To}_{i,j})+\frac{\ch_0\left(\pi_l(V_{i,j} ^{\To}) -\pi_0({V_{i,j}^{\To})}\right)}{p},
\\ &b_{i,j,l}= \rank(T_{i,j} \M_i^{-}) + \frac{\ch_0\left( \pi_l(V_{i,j} ^{-}) -\pi_0(V_{i,j}^{-})\right)}{p}.
\end{align}
Thus it is clear that $a_{i,j,l}=\tilde{a}_{i,j,l}$ and $b_{i,j,l}=\tilde{b}_{i,j,l}$. Therefore also $\rho_{i,j}=\tilde{\rho}_{i,j}$. It follows that the right hand side of equation \eqref{eq:leadingterm} is equal to $\tilde{\zeta}_{i,j}$. Thus we can conclude by applying Lemma \ref{lemmapthroot} applied to the elements $a=\ch(\tilde{\chi}_{i,j})$ and $b=\ch(\Omega(T_{i,j} \M)).$

\paragraph*{Second part.}  We now prove \eqref{eq:inducedlinebundle}. Define a $\To$-equivariant line bundle on $\tM_i$ by $\check{\LD}_i=(\varXi_i^{-1})^*(\LD_{\mid \M_i})$. We must show that $c_1(\check{\LD}_i)=\tilde{\omega}_i$. Recall that $\tilde{\n}_i$ is diffeomorphic to a certain moduli space of flat $G$-bundles on $\tilde{X}\setminus \D_i$. Through this identification, the moduli space $\tilde{\n}_i$ supports an induced symplectic form $\tilde{\Omega}_i$ as constructed in \cite{freedClassicalChernSimonsTheory1995}, whose cohomology class satisfy
\begin{equation}
    [\tilde{\Omega}_i]=\left( \tilde{\alpha}_i+\sum_{x \in \D_i} w_{i,2}(x) \delta_{i,x}\right)_{\mid \n_i}= \frac{\tilde{\omega_i}_{\mid \tilde{\n}_i}}{p}.
\end{equation}
It was observed in \cite{andersenWittenReshetikhinTuraev2013} that
\begin{equation} \label{eq:curvatureequation}
c_1(\check{\LD}_{i \mid \tilde{\n}_i})=\tilde{\omega}_{i \mid \tilde{\n}_i}=p[\tilde{\Omega}_i].
\end{equation}
The natural map $\Pic(\tM_i) \rightarrow \Pic(\tilde{\n}_i)$ is an isomorphism. This is proven in \cite{Roy2023} in the case of moduli spaces of stable parabolic Higgs bundles of coprime rank and degree and stable parabolic holomorphic vector bundles of coprime rank and degree, and the result generalizes easily to the cased of fixed determinant. Therefore, equation \eqref{eq:curvatureequation} implies the following identity in cohomology
\begin{equation} \label{eq:iii}
  c_1(\check{\LD}_i)=\tilde{\omega}_i.
\end{equation}
Thus it only remains to argue that this identity holds as an identity in $\To$-equivariant cohomology. 

Recall that $\To$-equivariant cohomology of a point is additively isomorphic to a polynomial ring of one formal variable, which we denote by $u$, of cohomological degree two. Recall that as $\M$ is a semi-projective variety in the sense of \cite{bialynicki-birulaTheoremsActionsAlgebraic1973}, we have that $H^*(\M)\cong H^*(\M^{\To})$ and by \cite{K84} the $\To$-equivariant cohomology ring $H^*_{\To}(\M)$ is isomorphic as a $\Q[u]$-module to the ring of polynomials in $u$ with coefficients in $H^*(\M)$, i.e. in symbols
\begin{equation}
    H^*_{\To}(\M) \cong H^*(\M) [u] \cong H^*(\tM^{\To})[u].
\end{equation} 
Thus, for any $\To$-equivariant line bundle $L \rightarrow \M$, and any component $Z \subset \M^{\To},$ the weight of the $\To$-action on $L_{\mid Z}$ is equal to an integer $b$ if and only if the coefficient of $u$ in $c_1(\LD)_{\mid Z}$ is equal to $b$. Similarly, $H^*_{\To}(\tM_i) \cong H^*(\tM_i^{\To})[u]$. Because of this isomorphism, we see that establishing \eqref{eq:inducedlinebundle} for every $i \in \mathcal{I}$ and $ j \in \mathcal{J}_i$ is equivalent to establing \eqref{eq:identificationofinducedTLibebundle} for every $i \in \mathcal{I}$.

Thus, we mush show that for every component $\tM_{i,j}$ of $\tM_i^{\To}$ the following holds in $\To$-equivariant cohomology
\begin{equation} \label{eq:curvaturecomponent}
    c_1(\check{\LD}_{i \mid \tM_{i,j} })=\tilde{\omega}_{i \mid \tM_{i,j}}.
\end{equation}
Because of the identity between complex line bundles given in equation \eqref{eq:iii}, we see that equation \eqref{eq:curvaturecomponent} is equivalent to the coefficient of $u$ being the same in $c_1(\LD)_{\mid \M_{i,j}}$ and in $\tilde{\omega}_{i \mid \tM_{i,j}}$. This fact was already established for the component corresponding to $\tilde{\n}_i$ (as the coefficient of $u$ is equal to zero on both sides of the equation). Thus, it remains to prove this fact for a component $\M_{i,j}$ with $j \in \mathcal{J}_{i,c}$ and $c \in C$.

Write $\alpha=c_1(\LD).$ By definition, we have that $\M_{i,j} \subset \mathcal{F}_c$, and by  \cite[Lemma 6.1]{hauselRelationsCohomologyRing2003} it holds that the coefficient of $\alpha_{\mid \mathcal{F}_c}$ is equal to $\chi+c$, where $\chi=2-2g$ is the Euler characteristic of the Riemann surface $X$. By adapting the proof of \cite[Lemma 6.1]{hauselRelationsCohomologyRing2003} one can show that the coefficient of $u$ in the restriction of $\tilde{\alpha}_i$ to $\tM_{i,j}$ is equal to the integer $\mu_{ij}$ defined to be
\begin{equation}
    \mu_{ij}= \tilde{\chi}-\deg(\mathcal{\D}_i)+\deg(\D_{ij})+l_j,
\end{equation}
where $\tilde{\chi}=2-2\tilde{g}$ is the Euler characteristic of $\X$, $l=l_j=(c-\deg(j))/p$  and $\D_{ij}$ is the divisor defined in \eqref{eq:Dij}. Similarly, the coefficient of $u$ in $\delta_{i,x}$ restricted to $\tM_{i,j}$ is equal to $\epsilon_{j,x}$. Recall $pw_{i,2}(x)=b_i(x)$ for each $x \in \D_i$ by the definitions of $w_i$ and $b_i$. Hence, the coefficient $c_{ij}$ of $u$ in $\tilde{\omega}_i$ restricted to $\tM_{i,j}$ is equal to
\begin{equation}
    c_{i,j}=\mu_{i,j}+\sum_{x \in \D_i} \epsilon_{j,x}b_i(x),
\end{equation}
and we must show that $c_{i,j}=\chi+c.$

Recall that since $X \rightarrow \X$ is a branched covering of prime order $p$, Hurwitz' formula simplifies to
\begin{equation} \label{eq:Hurwitz}
  \chi=  p\tilde{\chi}-(p-1)\lvert X^f\rvert.
\end{equation}
Using these facts, the expression for $p \deg(\D_{ij})=\deg(\pi^*(\D_{ij}))$ given by the right hand side of equation \eqref{eq:Dij}, and the fact that $pl_j+\deg(j)=c$ by definition of $j \in \mathcal{J}_{i,c}$, we see that the coefficient $c_{ij}$ of $u$ in $\tilde{\omega}_i$ restricted to $\tM_{i,j}$ is equal to $\chi+c$ by the following computation
\begin{align}
    c_{ij}&=p\mu_{ij}+\sum_{x \in \D_i} b_{i}(x) \epsilon_{j,x}
    \\&=p(\tilde{\chi}-\deg(\mathcal{\D}_i)+\deg(\D_{ij})+l)+\sum_{x \in \D_i} b_{i}(x) \epsilon_{j,x}
    \\ &=p\tilde{\chi}-p\lvert \mathcal{D}_i \rvert+pl+
 (\deg(j)+\sum_{x \in \D_i}(1-\epsilon_{j,x}b_i(x))-(p-1)\lvert X^f \setminus \mathcal{D}_i \rvert)
 \\& +\sum_{x \in \D_i} b_{i}(x) \epsilon_{j,x}
=(p\tilde{\chi}-(p-1)\lvert X^f \rvert)+(pl+\deg(j))= \chi+c.
\end{align}
This finishes the proof.
\end{proof}

\subsection{The Galois Case} \label{sec:Galois}
Recall the notation $\pi_\M: \M \times X \rightarrow \M$ for the natural projection onto the first factor. Assume that $X^f=\emptyset$. Then Hurwitz formula specialises to $\chi(X)=p \chi(\X)$, where $\chi(X)$ (resp. $\chi(\X))$) denotes the Euler characteristic of $X$ (resp. $\X$). Fix $\tilde{x}_0 \in \X$ and define 
\begin{equation}
\LD=\det((\pi_\M)_!(\Eb)[1])\bigotimes_{x \in \pi^{-1}(\tilde{x}_0)}\det(\Eb_{\mid \M\times\{x\} })^{1-g}.
\end{equation}
With notation as in the introduction, let $\tM$ be the moduli space of stable Higgs bundles on $\X$ of rank two and fixed determinant $\tilde{\Lambda}$ of odd degree. Let $\tilde{\Eb}\rightarrow \tM\times \X$ be the universal Higgs bundle. Consider the projection $\pi_{\tM} : \tM\times \X \rightarrow \tM$ and define
\begin{equation} 
\label{def:tildedet} \tilde{\LD}=\det(\tilde{\Eb}_{\mid \tM \times \{x_0\}})^{1-\tilde{g}} \otimes \det((\pi_{\tM})_!(\tilde{\Eb})[1]),
\end{equation} 
where $(\pi_{\tM})_!(\tilde{\Eb})[1]$ is the $1$-shift of the derived pushforward of $\tilde{\Eb}$.

\begin{proof}[Proof of Corollary \ref{thm:Galois}]

	In the case $X^f$ is empty, pullback with respect to the projection $\pi: X \rightarrow \X$ induces an isomorphism of $\To$-varities $\pi^*:\tM \rightarrow \M^f $. We write $\varXi=(\pi^*)^{-1}$ for the inverse isomorphism of $\To$-varieties. We have that $(\varXi^{-1})^*(\LD) \cong\tilde{\LD}^{\otimes p}$. We simplify the notation and write $\{\tilde{F}_a\}_{a\in A}$ for the components of $\tM^{\To}$. Let $a\in A$. Denote by $T_a \tM$ the restriction of $T\tM$ to $\tilde{F}_a$ and write ${\nu}_{a}=\rank(T_a\tM^{-})$. Observe that as $\D$ is empty, we have that $\tilde{\nu}_{a}=\nu_{a}$. Define the $t^{-1}$-series
	\begin{equation}
	\chi_{a}(t)= (-1)^{\tilde{\nu}_{a}} \int_{\tilde{F}_{a}} \ch(\omega(T_a\tM)\cdot \tilde{\LD}^k) \cdot \Td \tilde{F}_{a}.
 	\end{equation} Then $\chi_t(\tM,\tilde{\LD}^k)(t)=\sum_{a \in A} 	\chi_{a}(t)$ by the localization result of Wu \cite[Theorem 3.14]{wuInstantonComplexHolomorphic2003} recalled in Theorem \ref{thmWu} in this article. By Theorem \ref{ThmMain} it will therefore be sufficient to prove that for all $a\in A$ we have that
 	\begin{equation} \label{eq:sufficient}
 (-1)^{\tilde{\nu}_{a}} \int_{\tilde{F}_{a}} \ch( \tilde{\chi}_{a} \cdot\tilde{\LD}^{kp}) \cdot \Td \tilde{F}_{a} =	\chi_{a}(t^p).
 	\end{equation}
 	 We now analyze the left hand side of equation \eqref{eq:sufficient}. As $\D$ is empty, we find that
	\begin{equation}
	\tilde{\chi}_{a}= \omega(T_a\tM \cdot \Phi_p(\hat{\zeta}))\prod_{j=1}^{p-1} \lambda(T^* \tilde{F}_{a}\cdot \hat{\zeta}^j)^{-1}.
	\end{equation}
Let $u$ be a formal variable and recall the notation $s_u$ (resp. $\lambda_u$) for the operation of taking the graded sum of symmetric (resp. exterior) powers. Recall that if $\{b\}$ denotes the set of Chern roots of a general vector bundle $W$ then 
	\begin{equation} \label{eq:nyttigformel}
	\ch(s_u(W^*))=\ch(\lambda_{-u}(W^*))^{-1}=\prod_{b}(1-ue^{-b})^{-1}.
	\end{equation} Recall that if $c_{a}$ denotes the first Chern class of $\tilde{F}_{a}$, and if $\{z\}$ denotes the set of Chern roots of $T \tilde{F}_{a}$, then we have that $c_{a}=\sum_{z}z$ and
	\begin{equation} \label{eq:Todd}
	\Td \tilde{F}_{a}= \exp(c_{a}/2) \prod_{z} \frac{z}{e^{z/2}-e^{-z/2}}.
	\end{equation} 
	On the other hand, we have that
	\begin{align}
	& \label{eq:firstsimpli}  \prod_{j=1}^{p-1} \ch(\lambda(T^* \tilde{F}_{a}\cdot \hat{\zeta}^j )^{-1})=  \prod_{\alpha \in \mu_p^*}  \prod_{z} \frac{1}{ 1-\alpha e^{-z}} = \frac{\exp\left((p-1)c_l/2\right) }{\prod_{\alpha \in \mu_p^*}  \prod_{z} \left( e^{z/2}-\alpha e^{-z/2} \right)} 
	\\ &=  \exp\left((p-1)\frac{c_a}{2}\right)  \prod_{z}  \frac{\left( e^{z/2}- e^{-z/2} \right)}{\left( e^{pz/2}- e^{-pz/2} \right)}, \label{eq:simplification3}
	\end{align} 
	where, for equation  \eqref{eq:firstsimpli} we used \eqref{eq:nyttigformel}, and for equation \eqref{eq:simplification3} we used the algebraic identity 
	\begin{equation} \label{eq:algebra}
	Z^p-Z^{-p}=\prod_{\alpha \in \mu_p}(Z-\alpha Z^{-1}),
	\end{equation} where $Z$ is a formal variable. Define $d_{a}=\dim(\tilde{F}_{a})$. It follows, that if we write $\Td \tilde{F}_{a}=\Td \tilde{F}_{a}(z)$ emphasizing its dependence on the set of Chern-roots $\{z\}$, then we get from equation \eqref{eq:simplification3} and equation \eqref{eq:Todd} that the following identity holds
	\begin{equation} \label{iki}
\prod_{j=1}^{p-1} \ch(\lambda(T^* \tilde{F}_{a} \cdot \hat{\zeta}^j )^{-1}) \cdot	\Td \tilde{F}_{a}=  p^{-d_{l}}\Td \tilde{F}_{a}(pz).
	\end{equation}
Let $\{y\}$ denote the set of Chern-roots of $T_a \tM$ and write $\ch(\omega(T_a \tM))(t,y)$ to emphasize its dependence on $(t,y).$ By unwinding the definition of $\omega$ and applying again equation \eqref{eq:nyttigformel} and equation \eqref{eq:algebra}, we obtain 
	\begin{equation}
	\label{eq:uku}
\ch(\omega(T_a \tM \cdot \Phi_p(\hat{\zeta}))=\ch(\omega(T_a \tM))(t^p,py).
	\end{equation}
	Taking the product of equations \eqref{eq:uku} and \eqref{iki}, and exploiting homogeneity, we obtain the desired identity
	\begin{align}
	 &(-1)^{\tilde{\nu}_{a}} \int_{\tilde{F}_{a}} \ch(\tilde{\chi}_{a} \cdot  \tilde{\LD}^{kp}) \cdot \Td \tilde{F}_{a}  =
	 \\& (-1)^{\tilde{\nu}_{a}}p^{-d_{a}} \int_{\tilde{F}_{a}} \omega_{a}(T \tM)(t^p,py)\cdot \ch(\tilde{\LD}^p)^k \cdot  \Td
	 \tilde{F}_{a}(pz)
	 =\chi_{a}(t^p),
	\end{align}
	which was what we wanted to prove.
 \end{proof}

\section{Quantum Topology} \label{sec:QT} 
Recall the notation $M_f$ for the mapping torus of $f: \Sigma \rightarrow \Sigma,$ where $\Sigma$ is the underlying two-manifold of $X$.  

\begin{proof}[Proof of Corollary \ref{Cor:TopInv}]
	
Define the tuple of integers $(k_x)_{x \in X^f}$ by the conditions that for all $x \in X^f$ we have that $k_x=[-m_x]_p$, where $m_x$ was defined in Section \ref{sec:FIXEDLOCUS}. The mapping torus $M_f$  admits a unique Seifert fibration \cite{orlikSeifertManifolds1972} of type $\epsilon=o_1$ with Seifert invariants expressible in terms of the fixed point data \begin{equation} \label{eq:SeifertInvariants} (b,g,(a_1,b_1),...,(a_u,b_u))=(-p\Sigma_{x \in X^f} k^{-1}_{x},  \tilde{g} , (k_x,p)_{x \in X^f} ).\end{equation}
It is thus clear that $(D_{\Lambda},l_0)$ is determined by the Seifert invariants of $(M_f,K)$ as defined in the introduction.
\end{proof}

We now explain the refinement formula given in equation \eqref{eq:algquantumtop}.  We have that $\n^f$ is a subvariety of $\M^{\To\times \f}$. The torus $\To$ act trivially on the restriction of the determinant line bundle $\LD$ to $\n$, whereas $\To$ acts with non-trivial weight $t^{(2-2g)m}$ on $\LD$ restricted to $\mathcal{F}_m$ for each $m \in M$. Therefore the sum appearing in \eqref{eq:Main} corresponding to components of $\n^f$ is seen to give the constant term of $\chi_{\To}(\M,\LD^k,f)$. We have that $T_{\n} \M^{-}\cong T \n$, and $\det( T \n) \cong\LD^{2}$. As the rank of $T_{\n}\M^{-}$ is equal to $3g-3$, the constant is equal to $(-1)^{3g-3} $ times $\chi(\n,\LD^{k+2},f)$. Thus we obtain equation \eqref{eq:algquantumtop}. 

\subsection{Complex Quantum Chern-Simons, TQFT Invariants and The Automorphism Equivariant Hitchin Index} 

Classical Chern-Simons theory with complex gauge group was introduced in \cite{cheegerDifferentialCharactersGeometric1985}, and complex quantum Chern-Simons was studied in the physics literature in \cite{gukovThreedimensionalQuantumGravity2005,wittenQuantizationChernSimonsGauge1991}. In $\cite{wittenQuantizationChernSimonsGauge1991}$, the moduli space of flat complex connections on a two-manifold was quantized with respect to a family of real polarizations (different from the one considered in this article). This quantization was subsequently studied by the first author et al in the works \cite{andersenHitchinWittenConnectionComplex2014,andersenGenusoneComplexQuantum2021}. A rigorous mathematical construction of a full TQFT modelling complex quantum Chern-Simons theory have yet to be given, although it is widely believed that the quantum Teichmüller TQFT \cite{andersenQuantumTeichmullerTheory2014,AndersenKashaev18} of the first author and Kashaev is a part of it corresponding to the real gauge group $\PSL(2,\R)$.

As explained in detail in the introduction, the study of the automorphism equivaraint index  \eqref{eq:index} is motivated by TQFT axioms \cite{atiyahTopologicalQuantumField1988} and the prospect of a full three-dimensional TQFT that models quantum Chern-Simons theory with complex gauge group $\SL(2,\C)$. In particular, it is natural to conjecture that the index \eqref{eq:index} is closely connected to the partition function on the mapping torus with the presence of a (labelled) knot traced out by the marked point. This belief is supported by Corallary \ref{Cor:TopInv}. As such, Corollary \ref{Cor:TopInv} partially extends the bridge between algebraic geometry and quantum topology that exists via the realization of the TQFT representation of the mapping class group of a two-manifold through quantization of the moduli space of flat connections on said two-manifold. See e.g. \cite{andersenHitchinConnectionToeplitz2012,andersenGeometricConstructionModular2007,axelrodGeometricQuantizationChernSimons1991,hitchinFlatConnectionsGeometric1990,wittenQuantumFieldTheory1989} and the references therein.

\bibliography{AEHI}

\noindent
Jørgen Ellegaard Andersen \\
Center for Quantum Mathematics\\
Danish Institute for Advanced Studies \\
University of Southern Denmark\\
DK-5000 Odense C, Denmark\\
jea{\@@}.sdu.dk
\\\\
William Elbæk Mistegård  \\
Center for Quantum Mathematics\\
University of Southern Denmark\\
DK-5000 Odense C, Denmark\\
wem{\@@}imada.sdu.dk

\end{document}